\documentclass[a4paper, oneside, 10pt]{article} \usepackage[english]{babel}
\usepackage{amsfonts,amsmath,amsthm}
\usepackage[sort&compress,numbers]{natbib}
\title{\textbf{Relative topological entropy for actions of non-discrete groups on compact spaces in the context of cut and project schemes} } \author{T.~Hauser}
\let\epsilon=\varepsilon

\theoremstyle{definition}
\newtheorem{definition}{Definition}[section]

\newtheorem{theorem}[definition]{Theorem}
\newtheorem{proposition}[definition]{Proposition}
\newtheorem{lemma}[definition]{Lemma}
\newtheorem{corollary}[definition]{Corollary}

\newtheorem*{acknowledgement}{Acknowledgement}

\theoremstyle{remark}
\newtheorem{remark}[definition]{Remark}
\newtheorem{example}[definition]{Example}

\let\epsilon=\varepsilon
\let\phi=\varphi
\let\theta=\vartheta

\begin{document}
\maketitle

%\textbf{TO DO:}
%\begin{itemize}
%	%\item Link to other paper? -< dense sets
%	%\item Density and expansiveness?
%	%\item remark why it is important to use compact methods!
%	\item JLO reference
%\end{itemize}

\begin{abstract}

	In the study of aperiodic order via dynamical methods, topological entropy is an important concept. In this paper, parts of the theory, like Bowen's formula for fibre wise entropy or the independence of the definition from the choice of a Van Hove sequence, are extended to actions of several non-discrete groups. To establish these results, we will show that the Ornstein-Weiss lemma is valid for all considered groups which appear in the study of cut and project schemes. 
	
%	%The literature lacks an Ornstein-Weiss lemma and a systematic treatment of topological entropy of actions of non-discrete amenable unimodular groups, like for example $\mathbb{R}^d$. 
%	We prove an Ornstein-Weiss lemma for amenable unimodular groups containing a uniform lattice and show that averages along Van Hove nets can be obtained by averaging inside the lattice. We use this result to introduce relative topological entropy for actions of amenable unimodular groups that contain a uniform lattice and show that Bowens formula for relative topological entropy is satisfied.
\end{abstract}



\section{Introduction}

%L. Bowen \cite{bowen2014entropy}
%Introduction: \cite{bowen2017brief}

Aperiodic order, an intermediate concept between order and disorder, has attracted a lot of attention over the last three decades in the fields of physics, geometry, number theory and harmonic analysis \cite{shechtman1984metallic,axel1995beyond,patera1998quasicrystals,	baake2000directions,baake2004dynamical,baake2007pure,baake2015ergodic,baake2016spectral}. The construction of aperiodic point sets via cut and project schemes was pioneered by Yves Meyer in his famous monograph on ''Algebraic numbers and harmonic analysis''. For details see \cite[Chapter II.5]{meyer1972algebraic}. 
	A \emph{cut and project scheme (CPS)} is a triple $(G,H,\Lambda)$, where $G$ and $H$ are locally compact amenable groups\footnote{%Let $G$ be a topological group. %$G$ is called \emph{$\sigma$-compact}, whenever there is a countable cover of $G$ by compact sets. Furthermore 
	 A topological group $G$ is called \emph{locally compact}, whenever every neighbourhood of some element contains a compact neighbourhood of this element. For the notion of amenability see Subsection \ref{sec:Amenable+VanHove}.}, 
$\Lambda$ is a uniform lattice\footnote{A discrete subgroup $\Lambda$ of a locally compact group $G$ is called a \emph{uniform lattice}, whenever it is \emph{co-compact}, i.e.\ whenever $G\big/\Lambda$ is compact.}
 in $G\times H$ and the projections $\pi_G$ and $\pi_H$ satisfy the following properties. The restriction $\pi_G\big|_\Lambda$ is injective and $\pi_H(\Lambda)$ is dense in $H$.  
 Then $G$ is called the \emph{physical space} and $H$ is referred to as the \emph{internal space} of $(G,H,\Lambda)$. % Note that the $\sigma$-compactness of $G$ and $H$ imply $\Lambda$ to be countable. 
	Given a relatively compact subset $W\subseteq H$ with nonempty interior, usually called a \emph{window} in this context, 
such a CPS produces a subset of G via $\omega:=\pi_G(\Lambda\cap(G\times W))$. Subsets of $G$ that arise by this construction are called \emph{model sets}.
These sets are aperiodic, but have a longe range order, due to their algebraic origin. 
For further details and references on these notions see \cite[Chapter 7]{baake2013aperiodic}. CPS are usually studied under the assumption of commutativity. 	
	Nevertheless recent interest in the non commutative case, for example in \cite{bjorklund2018aperiodic} motivated us to omit the assumption of commutativity in our definition of CPS. 
	
	Model sets can be studied by methods of dynamical systems. One first introduces a compact Hausdorff topology on the set of all closed subsets of $G$ and shows that the set of all translations $\{\omega g;\, g\in G\}$ is a pre-compact subset. Denote by $X$ the closure of $\{\omega g;\, g\in G\}$. One then shows that $G\times X\ni (g,M)\mapsto M g:=\{mg;\, m\in M\}$ is a dynamical system, referred to as the \emph{Delone dynamical system} of $\omega$. For details on this construction see \cite{baake2004dynamical} in combination with \cite{baake2013aperiodic}. 
	In the study of model sets it is natural to study the topological entropy of this dynamical system as a measure of ''complexity'' of $\omega$ \cite{baake2007pure,baake2015ergodic}. 
	
	 It is thus natural to ask for the validity of analogues of statements from the theory of entropy of $\mathbb{Z}$-actions.
	In \cite{OrnsteinandWeiss,HuangandYeandZhang, Yan, ZhengandEChen, zhou2016tail} such analogues are proven for actions of countable discrete amenable groups $G$. 		 	
	 Nevertheless in the context of Delone dynamical systems $G$ is typically not discrete and we will see in Example \ref{exa:CPS} that there can be choices of $G$ that do not contain uniform lattices. In fact the absence of the possibility to restrict to discrete subgroups was one of the motivations of Meyer to construct CPS \cite[Chapter II]{meyer1972algebraic}.
	  The references known to us for a systematic treatment of entropy theory of actions of non discrete groups are \cite{Tagi-Zade,schneider2015topological}. Both notions are not equivalent for actions of $\mathbb{R}^d$ with $d\geq 2$ and we will focus on the notion of Tagi-Zade, which is used in the context of aperiodic order \cite{baake2007pure}. 
	  Tagi-Zade presents entropy theory of $\mathbb{R}^d$ actions, but parts like fibre wise entropy, relative topological entropy, the freedom in the averaging in the definition of entropy are not addressed. Furthermore \emph{Bowen's formula}, which states that the topological entropy of an action is less than the sum of the topological entropy of a factor and the relative topological entropy of the factor map \cite{bowen1971entropy}, are not considered.  
	 These parts of the theory are addressed in the study of aperiodic order in \cite{baake2007pure,HuckandRichard,JaegerLenzandOertel} and the importance of Bowen's formula for $\mathbb{R}^d$ actions comes up in \cite[Remark 2.9.]{fuhrmann2018irregular} and \cite[Lemma 4.1.(ii)]{fuhrmann2018irregular}. 
	 
	In order to define topological entropy of discrete amenable groups one uses a technique referred to as ''Ornstein-Weiss lemma''  \cite{OrnsteinandWeiss,ward1992abramov,lindenstrauss2000mean,Yan}. In order to explain this in more detail we need the following notions. 
 %For the notion of amenable groups or van-Hove nets, see subsection \ref{sec:Amenable+VanHove} below. 
	Let $G$ be an amenable group $G$. 
	Denote by $\mathcal{K}(G)$ the set of all compact subsets of $G$.
	A function $f\colon \mathcal{K}(G)\to \mathbb{R}$ is called \emph{subadditive}, if for all disjoint $A,B\in \mathcal{K}(G)$ there holds
	$f(A\cup B)\leq f(A)+f(B).$
Furthermore $f$ is said to be \emph{right invariant}, if for all $A\in \mathcal{K}(G)$ and for all $g\in G$ there holds
		$f(Ag)= f(A).$
A function  $f$ is  called \emph{monotone}, if for all $A,B\in \mathcal{K}(G)$ with $A\subseteq B$ there holds
		$f(A)\leq f(B).$ 
	$G$ is said to \emph{satisfy the Ornstein-Weiss lemma}, if for any subadditive, right invariant and monotone function $f\colon \mathcal{K}(G)\to \mathbb{R}$ the limit 
	\begin{align}\label{equ:OWlimit}
		\lim_{i\in I} \frac{f(A_i)}{\mu(A_i)}
	\end{align}
	exists, is finite and does not depend on the choice of the Van Hove net $(A_i)_{i\in I}$ in $G$ .  
	
	The ideas behind this technique go back to \cite{OrnsteinandWeiss,ward1992abramov,HuangandYeandZhang,gromov1999topological,lindenstrauss2000mean}. In M. Gromovs extremly influential work \cite{gromov1999topological} a sketch of a proof is presented for very general groups, but a rigorous proof is given only in the context of discrete amenable groups. For details see \cite{Krieger,Yan,FeketesLemma} %Note that the arguments also generalize to the context of amenable discrete semigroups, as presented in \cite{FeketesLemma}. 
	% The problem with the non discrete version seems to be hidden at the end of the sketch in \cite{gromov1999topological}. For details see 
	 and Remark \ref{rem:Ornsteinweissgroups}(ii). We thus give a rigorous proof of the Ornstein-Weiss lemma for groups arising in our context. The next theorem shows that all compactly generated abelian groups satisfy the Ornstein-Weiss lemma.

	 \begin{theorem}\label{theint:OrnsteinWeisslemma}%\label{rem:restrictionof_h_tolattice_OW_Formula}
		Every amenable group containing a uniform lattice satisfies the Ornstein-Weiss lemma.
	\end{theorem} 

	If $(G,H,\Lambda)$ is a CPS we know that $G\times H$ contains a uniform lattice and in order to show that $G$ satisfies the Ornstein-Weiss lemma one could hope that the properties of a CPS also imply that $G$ contains a uniform lattice. However in Example \ref{exa:CPS}, which is a special case of examples, studied by Meyer in \cite[Chapter II.10]{meyer1972algebraic}, we present a CPS with a physical space that does not contain any uniform lattice. %In fact this physical space is the set of $p$-adic numbers and the lack of uniform lattices in this set was one of the reasons to introduce 	
	Nevertheless we obtain that physical spaces of CPS satisfy the Ornstein-Weiss lemma from the next result.   

\begin{theorem}\label{theint:ProductandOW}
	If $G$ and $H$ are amenable groups such that $G\times H$ satisfies the Ornstein-Weiss lemma, then $G$ and $H$ satisfy the Ornstein-Weiss lemma.
\end{theorem}

\begin{corollary}\label{cor:CPSimpliesOW}
	Let $G$ be a locally compact amenable group. If there is a CPS, such that $G$ is the respective physical space, then $G$ satisfies the Ornstein-Weiss lemma. 
\end{corollary}	

	Naturally the question arises how restrictive the existence of such a CPS is. In fact in the commutative case it is shown by Y. Meyer, that $G$ is the physical space of a CPS, iff it is the physical space of a CPS with an euclidean internal space \cite[Chapter II.5.10]{meyer1972algebraic}, and this holds iff there exists a Meyer set in $G$ \cite[Chapter II.14]{meyer1972algebraic}. A \emph{Meyer set}\footnote{Note that these sets are the ''harmonious and relatively dense'' sets in \cite{meyer1972algebraic}.} is a discrete subset $\omega\subseteq G$ such that there are a finite set $F\subseteq G$ and  a compact set $K\subseteq G$ that satisfy $K\omega =G$ and $\omega\omega^{-1}\subseteq F\omega$. Examples of Meyer sets are all model sets. 
	
	A Meyer set $\omega$ that is symmetric $(\omega=\omega^{-1})$ and that contains the neutral element is called an \emph{approximate uniform lattice} and the arguments from \cite{meyer1972algebraic} also show that a locally compact abelian group $G$ is a physical space of a CPS iff $G$ contains a uniform approximate lattice \cite{bjorklund2018approximate}. This equivalence remains valid also in the context of connected nilpotent Lie groups  \cite{machado2018approximate}, but the existence of a CPS under the assumption of the existence of a Meyer set seems open for general locally compact groups. An example of a metrizable and separable locally compact abelian group $G$ that contains no Meyer set and therefore is not a physical space of a CPS is given in \cite[Chapter II.11]{meyer1972algebraic}. 
	
	Nevertheless we are interested in groups $G$ that contain Meyer sets and for those we can now define (relative) topological entropy for some action $\phi$ of $G$ on a compact metric space $(X,d)$. 
	For a compact subset $A\subseteq G$ define the Bowen metric for $x,y\in X$ as follows
	\[d_A(x,y):=\max_{g\in A} d(\alpha(g,x),\alpha(g,y)).\]
	Furthermore for $M\subseteq X$ and $\epsilon>0$ we denote the minimum cardinality of an open cover of $M$ consisting of sets of $d_A$-diameter\footnote{The \emph{$d$-diameter} of a set $M\subseteq X$ is defined by $\sup_{(x,y)\in M^2}d(x,y)$.	
	} strictly less than $\epsilon$ by $\operatorname{cov}_M[d_A<\epsilon]$. One then shows that $\mathcal{K}(G)\ni A\mapsto \log(\operatorname{cov}_M[d_A<\epsilon])$ is a monotone, right invariant and subadditive mapping and thus the Ornstein Weiss lemma can be applied to yield the existence of the following limit, as well as the independence from the choice of a Van Hove net. For some Van Hove net $(A_i)_{i\in I}$ one defines the \emph{topological entropy} of $\phi$ as 
	\[\operatorname{E}(\phi):=\sup_{\epsilon>0}\lim_{i\in I} \frac{\log(\operatorname{cov}_X[d_{A_i}<\epsilon])}{\mu(A_i)}.\]
	Furthermore if $p\colon X\to Y$ is a factor map onto some action $\psi\colon G\times Y\to Y$ one defines with a similar argument the \emph{relative topological entropy (of $p$)} as
	\[\operatorname{E}(\phi\overset{p}{\rightarrow}\psi):=\sup_{\epsilon>0}\lim_{i\in I} \frac{\log(\sup_{y\in Y}\operatorname{cov}_{p^{-1}(y)}[d_{A_i}<\epsilon])}{\mu(A_i)}.\]
	In particular there holds $\operatorname{E}(\phi\overset{p}{\rightarrow}\psi)=\operatorname{E}(\phi)$, whenever $Y$ is a single point. 
	
	It is standard to define the topological entropy of an action of $\mathbb{R}$ as the restriction to the action of $\mathbb{Z}$. We present next that in a similar way one obtains the relative topological entropy of an action as the scaled entropy of the restricted action to certain model sets and in particular any uniform lattice. 
	This allows to transfer several Theorems proven for discrete amenable groups to our context. 	
	
	To formulate the exact statement we need the following notions. We denote by $|F|$ the cardinality of a set $F$ and by $\mu$ the Haar measure on $G$. Furthermore we say, that a discrete subset $\Lambda\subseteq G$ has \emph{a well defined uniform density}, if $\operatorname{dens}(\Lambda):=\lim_{i\in I}\frac{|\Lambda \cap A_i|}{\mu(A_i)}$ exists, is finite and is independent from the choice of the Van Hove net $(A_i)_{i\in I}$. We refer to $\operatorname{dens}(\Lambda)$ as the \emph{uniform density} of $\Lambda$. Note that every uniform lattice $\Lambda$ has a well defined uniform density, which is given by $\operatorname{dens}(\Lambda):=\mu(C)^{-1}$, where $C$ denotes a fundamental domain of $\Lambda$. See Section \ref{sec:prelims} for details on this notion. Furthermore in locally compact abelian groups all regular\footnote{A model set is said to be \emph{regular}, if the Haar measure of the topological boundary of the corresponding window is $0$.} model sets have a well defined uniform density \cite[Corollary 15.1]{strungaru2015almost}. A subset $\Lambda\subseteq G$ is called \emph{relatively dense}, if there holds $K\Lambda =G$ for some compact subset $K\subseteq G$. Note that all model sets are relatively dense. 
	Denote for any map $f\colon A \to B$ and any subset $M\subseteq A$ by $f\big|_M$ the restriction $f\big|_M\colon M\to B\colon a\mapsto f(a)$. 
The following statement is contained in the statement of Theorem \ref{the:linkoflatticetogroup} in Section \ref{sec:Latticerestriction}. 

 \begin{theorem}\label{theint:linkoflatticetogroup}
 Let $\phi$ be an action of $G$ on a compact metric space $X$. Let furthermore $\psi$ be a factor of $\phi$ via factor map $p\colon X\to Y$. Let $\Lambda$ be a relatively dense subset of $G$ and let $(A_i)_{i\in I}$ be a Van Hove net. Set $F_i:=A_i\cap \Lambda$.
	\begin{itemize}
		\item[(i)] If $\Lambda$ has a well defined uniform density $\operatorname{dens}(\Lambda)$, then there holds 
\[\operatorname{E}(\phi\overset{p}{\to}\psi)=\operatorname{dens}(\Lambda)\sup_{\epsilon>0}\lim_{i\in I} \frac{\log(\sup_{y\in Y}\operatorname{cov}_{p^{-1}(y)}[d_{F_i}<\epsilon])}{|F_i|}\]
	\item[(ii)] If $\Lambda$ is a uniform lattice, then there holds
	\[\operatorname{E}(\phi\overset{p}{\to}\psi)=\operatorname{dens}(\Lambda)\operatorname{E}\left(\phi\big|_{\Lambda\times X}\overset{p}{\to} \psi\big|_{\Lambda\times Y}\right).\]
	\end{itemize}

%Let $\phi$ be an action of a group $G$ that satisfies the Ornstein-Weiss lemma on a compact Hausdorff space $X$. Let furthermore $\psi$ be a factor of $\phi$ via factor map $p\colon X\to Y$. Let $\Lambda$ be a uniform lattice in $G$ with fundamental domain\footnote{See section \ref{sec:prelims} for a definition.} $C$. Then 
%	\[\mu(C)\operatorname{E}(\phi\overset{p}{\to}\psi)=\operatorname{E}\left(\phi\big|_{\Lambda\times X}\overset{p}{\to} \psi\big|_{\Lambda\times Y}\right).\]
\end{theorem} 
	
	We extend Bowens formula beyond $\mathbb{Z}$ actions in the next result to actions of compactly generated locally compact abelian groups such as $\mathbb{R}^d$ and $\mathbb{Z}^d$, which are the most common choices in the study of cut and project schemes. 
\begin{theorem} \label{theint:Bowenstheoremgeneral}
		Let $\phi$, $\psi$ and $\rho$ be actions of an amenable group containing a countable uniform lattice on compact Hausdorff spaces $X$, $Y$ and $Z$ respectively. Let $\psi$ be a factor of $\phi$ via factor map $p$ and $\rho$ be a factor of $\psi$ via factor map $q$. Then there holds 
		\begin{align*}
		\max\{\operatorname{E}(\phi \overset{p}{\to}\psi),\operatorname{E}(\psi \overset{q}{\to}\rho)\} \leq \operatorname{E}(\phi \overset{q\circ p}{\to}\rho) 
		\leq \operatorname{E}(\phi \overset{p}{\to}\psi)+\operatorname{E}(\psi \overset{q}{\to}\rho). 
		\end{align*}
	\end{theorem}

The article is structured as follows. In Section \ref{sec:prelims} we fix some notion. Section \ref{sec:Ornstein-Weiss groups} is devoted to the proof of the Theorems \ref{theint:OrnsteinWeisslemma} and \ref{theint:ProductandOW}. In Section \ref{sec:Entropy theory for Ornstein-Weiss groups} we present basic results from the theory of relative topological entropy. As the mentioned topology in the construction of Delone dynamical systems is naturally defined via the notion of a uniformity \cite{Schlottmann,baake2004dynamical} and the corresponding arguments are well known  \cite{OrnsteinandWeiss,ward1992abramov,ollagnier2007ergodic,HuangandYeandZhang, weiss2003actions, Yan, ZhengandEChen, zhou2016tail}, we took the freedom to follow an idea from \cite{Hood,ollagnier2007ergodic,DikranjanandSanchis,YanandZeng} and use the language of uniformities in order to prove the results for actions on compact Hausdorff spaces. 
%	As an application of the independence of the definition of relative topological entropy from the choice of a Van-Hove net we will present that the approach to relative topological entropy simplifies for positive expanding systems similarly to the case of actions of continuous maps, like considered in \cite{BrinandStuck}. 	
	The approach via uniformities also yields tools to provide the equivalence of the definitions given in \cite{Tagi-Zade}, our and the classical definitions. %, which will be done in Section \ref{sec:Actions on metric and topological spaces}. 
	In particular we show that the implicit dependencies on Van Hove nets in \cite{baake2007pure,HuckandRichard,JaegerLenzandOertel} can be dropped.
	Theorem \ref{theint:linkoflatticetogroup} is proven in Section \ref{sec:Latticerestriction}.
	In Section \ref{sec:The Bowen entropy formula for actions of groups that contain a uniform lattice} we present a proof of Theorem \ref{theint:Bowenstheoremgeneral} and properties of the factor map $p$ under which we obtain $\operatorname{E}(\phi)=\operatorname{E}(\psi)$.

\section{Preliminaries}\label{sec:prelims}

In this section we provide notion and background on topological groups, uniformities, topological dynamical systems, amenable groups, Van Hove nets and uniform lattices.

\subsection{Topological groups}
%(OWG implies amenable implies  unimodular implies locally compact. 
Consider a group $G$. We write $e_G$ for the neutral element in $G$. For subsets $A,B\subseteq G$ the \emph{Minkowski product} is defined as 
$AB:=\{ab;\, (a,b)\in A \times B\}.$
For $A\subseteq G$ and $g\in G$ we denote $Ag:=A\{g\}$, $gA:=\{g\}A$, $A^c:=G\setminus A$ and the \emph{Minkowski inverse} $A^{-1}:=\{a^{-1};\, a\in A\}$. We call $A\subseteq G$ \emph{symmetric}, if $A=A^{-1}$. In order to omit  brackets, we will use the convention, that the inverse and the complement are stronger binding than the Minkowski product, which is stronger binding than the remaining set theoretic operations. Note that the complement and the inverse commute, i.e. $(A^c)^{-1}=(A^{-1})^c$. %Thus for $A\subseteq G$ there holds for example $AA^{-1}\setminus AA^c=(A(A^{-1}))\setminus (A(A^c))$.

A \emph{topological group} is a group $G$ equipped with a $T_1$-topology\footnote{
A topology is called $T_1$, if for any two distinct points $g,g'\in G$ there is an open neighbourhood of $g$ that does not contain $g'$.
}
 $\tau$, such that the multiplication $\cdot\colon G\times G \to G$ and the inverse function $(\cdot)^{-1}\colon G\to G$ are continuous. %Furthermore we assume a topological group to be $T_0$, i.e. that for any two distinct points $g,g'\in G$ there is an open neighbourhood of $g$ that does not contain $g'$.
 With our definition every topological group is regular, hence Hausdorff, as shown in \cite[Theorem 4.8]{HewittandRoss}. An \emph{isomorphism of topological groups} is a homeomorphism that is a group homomorphism as well. We write $\overline{A}$ for the {closure} and $\operatorname{int}(A)$ for the {interior} of a subset $A\subseteq G$. By $\mathcal{K}(G)$ we denote the set of all non-empty compact subsets of $G$. %Furthermore we abbreviate the term abelian locally compact group by \emph{LCAG}. 

If $G$ is a locally compact group, a \emph{Haar measure} on $G$ is a non zero regular Borel measure $\mu$ on $G$, which satisfies $\mu(gA)=\mu(A)=\mu(Ag)$  for all $g\in G$ and all Borel sets $A\subseteq G$. $G$ is called \emph{unimodular}, if it admits a Haar measure.
There holds $\mu(U)>0$ for all non empty open $U\subseteq G$ and  $\mu(K)<\infty$ for all compact $K\subseteq G$. 
A Haar measure is unique up to scaling, i.e.\ if $\mu$ and $\nu$ are Haar measures on $G$, then there is $c>0$ such that $\mu(A)=c\nu(A)$ for all Borel measurable sets $A\subseteq G$. %(see \cite[Theorem 10.14]{Folland}).
If nothing else is mentioned, we denote a Haar measure of a topological group $G$ by $\mu$. If $G$ is discrete we equip $G$ with the counting measure $A\mapsto|A|$, which is a Haar measure. Other examples of unimodular groups are all locally compact abelian groups or the Heisenberg group, as presented in Subsection \ref{sub:Uniformlattices} below.  
For reference see \cite{Folland,deitmar2014principles}.

\subsection{Compact Hausdorff uniform spaces}

%\subsubsection{Binary relations}
	Let $X$ be a set. 
	A \emph{binary relation on $X$} is a subset of $X\times X$. For binary relations $\eta$ and $\kappa$ on $X$ we denote the \emph{inverse}  
$\eta^{-1}:=\{(y,x);\, (x,y)\in \eta \}$, the \emph{composition} $\eta \kappa:=\{(x,y);\, \exists z\in X : (x,z)\in \eta \text{ and } (z,y)\in \kappa\}$ and $\eta[x]:=\{y\in X;\, (y,x)\in \eta\}$.
A binary relation is called \emph{symmetric}, if $\eta=\eta^{-1}$. 
%For $\eta\subseteq X\times X$ and $x\in  X$ we write %$[x]\eta:=\{y\in X;\, (x,y)\in \eta\}$ and $\eta[x]:=[x]\eta^{-1}$. 

%For $M\subseteq X$ we denote $[M]\eta:=\bigcup_{x\in M}[x]\eta$ and $\eta[M]:=[M]\eta^{-1}$. %We use the convention that these operations are stronger binding than set theoretic operations
%We write furthermore 
%\[[M]\eta:=\{x\in X;\, \exists m\in M: (m,x)\in \eta\}\] 
%and $\eta[M]:=[M]\eta^{-1}$ for $M\subseteq X$ and $\eta\subseteq X\times X$ and abbreviate $[x]\eta:=[\{x\}]\eta$ and $\eta[x]:=\eta[\{x\}]$ for $x\in X$. 

%\subsubsection{Uniform spaces}

	For a compact Hausdorff space $X$, we denote the \emph{diagonal} $\Delta_X:=\{(x,x);\, x\in X\}$ and call a neighbourhood of $\Delta_X$ in $X^2$ an \emph{entourage (of $X$)}. The set of all entourages of $X$ is referred to as the \emph{uniformity of $X$} and usually denoted by $\mathbb{U}_X$. In this context we refer to $(X,\mathbb{U}_X)$ as a \emph{compact Hausdorff uniform space}. 
	Note that one can define general ''uniform spaces'', but as we are only interested in compact Hausdorff spaces, this definition works for us. For details and the general definition we recommend \cite{Kelley}. Note that we obtain our definition to be a restriction of the general definition from \cite[Theorem 6.22]{Kelley} and \cite[Theorem 32.3]{Munkres}. %A \emph{uniformity} for a set $X$ is a non-empty family $\mathbb{U}_X$ of reflexive binary relations such that $\mathbb{U}_X$ is invariant under composition, inversion and intersection and such that for $\eta\in \mathbb{U}_X$ also $\kappa\in \mathbb{U}_X$, whenever $\eta\subseteq\kappa$. 
%\begin{itemize}
%	\item[(a)] each member of $\mathbb{U}_X$ contains the diagonal $\Delta_X$;
%	\item[(b)] if $\eta \in \mathbb{U}_X$, then $\eta^{-1}\in \mathbb{U}_X$;
%	\item[(c)] if $\eta \in \mathbb{U}_X$, then there is $\kappa \in \mathbb{U}_X$ such that $\kappa \kappa \subseteq \eta$;
%	\item[(d)] if $\eta$ and $\kappa$ are members of $\mathbb{U}_X$, then so is $\eta \cap \kappa$; and
%	\item[(e)] if $\eta \in \mathbb{U}_X$ and $\eta \subseteq \kappa \subseteq X\times X$, then $\kappa \in \mathbb{U}_X$.
%\end{itemize}
%The pair $(X,\mathbb{U}_X)$ is called a \emph{uniform space} and the members of $\mathbb{U}_X$ are called \emph{entourages}. 
To obtain some geometric intuition for $\eta\in \mathbb{U}_X$ we say that \emph{$x$ is $\eta$-close to $y$}, whenever $(x,y)\in \eta$. 
This notion is symmetric iff $\eta$ is symmetric. We think of two elements to be ''very close'', whenever the pair is $\eta$-close for ''many'' entourages $\eta$.  
 %If $x$ is $\eta$-close to $x'$ and $x'$ is $\eta$-close to $x$, we say that \emph{$x$ and $x'$ are $\eta$-close}.
Note that if $x$ is $\eta$-close to $y$ and $y$ is $\kappa$-close to $z$, then $x$ is $\eta \kappa$-close to $z$.

%If $(X,\mathbb{U}_X)$ is a uniform space the corresponding \emph{uniform topology} $\mathcal{T}_X$ consists of all subsets $U\subseteq X$ such that for each $x\in U$ there exists $\eta\in \mathbb{U}_X$ with $\eta[x]\subseteq U$. Topological terminology in the context of uniform spaces refers to this topology. 
A subfamily $\mathbb{B}_X\subseteq \mathbb{U}_X$ is called a \emph{base for $\mathbb{U}_X$}, if every entourage contains a member of $\mathbb{B}_X$. An entourage $\eta \in \mathbb{U}_X$ is called \emph{open} (or \emph{closed}), whenever it is open (or closed) as a subset of $X\times X$.  Note that the family of all open and symmetric entourages of $X$ forms a base of the uniformity of $X$. %see \cite[Theorem 6.6]{Kelley}
	If $(X,d)$ is a metric space we denote $[d<\epsilon]:=\{(x,y)\in X \times X;\, d(x,y)<\epsilon \}$ for $\epsilon>0$.  
Then $\mathbb{B}_d:=\{[d<\epsilon];\, \epsilon>0\}$ is a base for the uniformity of the corresponding  topological space $X$. 
% $\mathbb{U}_X$ consisting of all binary relations that contain $[d<\epsilon]$ for some $\epsilon>0$. The corresponding topology is the topology of open sets with respect to $d$.  
Note that $x$ is $[d<\epsilon]$-close to $y$, iff $d(x,y)<\epsilon$.

\subsection{Actions of a group on a topological space}
%Flows

Let $G$ be a topological group and $X$ be a topological space. A continuous map $\phi\colon G\times X \to X$ is called an \emph{action of $G$ on $X$} (also \emph{dynamical system} or \emph{flow}), whenever $\phi(e_G,\cdot)$ is the  identity on $X$ and for all $g,g'\in G$ there holds $\phi(g,\phi(g',\cdot))=\phi(gg',\cdot)$. We write $\phi^g:=\phi(g,\cdot)\colon X\to X$ for all $g\in G$. %Thus $\phi^{e_G}$ is the identity on $X$ and $\phi^g\circ \phi^{g'}=\phi^{gg'}$ for all $g,g'\in G$. 
In this context $X$ is called the \emph{phase space} of the action.  
If $\phi$ and $\psi$ are actions of a topological group $G$ on topological spaces $X$ and $Y$ respectively, we call a surjective continuous map $p\colon X\to Y$ a \emph{factor map}, if $p\circ \phi^g=\psi^g\circ p$ for all $g\in G$. We then refer to $\psi$ as a \emph{factor} of $\phi$ and write $\phi \overset{p}{\to}\psi$. If $p$ is in addition a homeomorphism, then $p$ is called a \emph{topological conjugacy} and we call $\phi$ and $\psi$ \emph{topologically conjugate}.

\subsection{Amenable groups and Van Hove nets}\label{sec:Amenable+VanHove}

A partially ordered set $(I,\geq)$ is said to be \emph{directed}, if $I$ is not empty
and if every finite subset of $I$ has an upper bound. A map $f$ from a directed set $I$ to a set $X$ is called a \emph{net} in $X$. We  also write $x_i$ for $f(i)$ and $(x_i)_{i\in I}$ for $f$.
A net $(x_i)_{i\in I}$ in a topological space $X$ is said to \emph{converge to $x\in X$}, if for every open neighbourhood $U$ of $x$, there exists $j\in  I$ such that $x_i\in U$ for all $i \geq j$. In this case we also write $\lim_{i\in I} x_i = x$. 
%Let $(x_i)_{i\in I}$ and $(y_j)_{j\in J}$ be nets in a set $X$. The net $(y_j)_{j\in J}$ is called a \emph{subnet} of $(x_i)_{i\in I}$ if there exists a function $\phi\colon J \to I$ such that $y_j=x_{\phi(j)}$ for $j\in J$ and such that for all $i \in I$ there is $m \in J$ with the property that, if $j \geq m$ then $\phi(j) \geq i$.
For a net $(x_i)_{i\in I}$ in $\mathbb{R}\cup \{ -\infty,\infty\}$, we define 
$\limsup_{i\in I} x_i:=\inf_{i\in I}\sup_{j\geq i}x_j$
and similarly $\liminf_{i\in I} x_i$. Note that $(x_i)_{i\in I}$ converges to $x\in \mathbb{R}\cup\{-\infty,\infty\}$, iff there holds $\limsup_{i\in I}x_i=x=\liminf_{i\in I}x_i$. 
For more details, see \cite{DunfordandSchwartz} and \cite{Kelley}.

	Let $G$ be a unimodular group. For $K,A\subseteq G$ we define the \emph{$K$-boundary of $A$} as
\[\partial_K A:=K\overline{A}\cap K \overline{A^c}.\]
	We use the convention, that the Minkowski operations and the complement are stronger binding than the operation of taking the $K$-boundary and that the set theoretic operations (except from forming the complement) are weaker binding. From the definition we obtain that $K\mapsto \partial_K A$ is monotone. 
	%Note that $K \overline{A}$ is the set of $g\in G$ such that $K^{-1}g\cap \overline{A}$ is not empty. 
Note that $\partial_K A$ is the set of all elements $g\in G$ such that $K^{-1}g$ intersects both $\overline{A}$ and $\overline{A^c}$. 
\begin{lemma}\label{lem:VHNbasicrules}
	For compact subsets $K,L,A\subseteq G$ there holds 
	\begin{itemize}
		\item[(i)] $L \partial_K A \subseteq\partial_{LK}A$ and $\partial_K LA \subseteq \partial_{KL} A$.
		\item[(ii)] $LA\subseteq A\cup \partial_LA$, whenever $e_G\in L$.
	\end{itemize}
\end{lemma}
\begin{proof}
	Straight forward arguments show (ii) and the first statement in (i). To see $\partial_K LA\subseteq \partial_{KL} A$ we compute $\overline{(LA)^c}\subseteq \overline{(lA)^c}=l\overline{A^c}\subseteq L\overline{A^c}$ for any $l\in L$ and obtain 
	$\partial_K LA \subseteq K\overline{LA}\cap K\overline{(LA)^c}\subseteq KL\overline{A}\cap KL\overline{A^c}=\partial_{KL} A$.
\end{proof}

%
%
%\begin{remark} \label{remark:basicrelationsoftheboundary}
%	\begin{itemize}
%		\item[(a)] For $A,B,C\subseteq G$ there holds $A(B\cap C)\subseteq (AB)\cap(AC)$ and 
%		$A(B\cup C)\subseteq (AB)\cup(AC)$.
%		\item[(b)] For $K,L,A\subseteq G$ there holds $L \partial_K A \subseteq\partial_{LK}A$.
%		\item[(c)] For $K,L,A\subseteq G$ there holds $\partial_K A\subseteq \partial_{L} A$, whenever $K\subseteq L$.
%	\end{itemize} 
%\end{remark}

%\begin{proof}
	%To show (a) let $g\in A(B\cap C)$ and note that there are $a\in A$ and $d\in B\cap C$ with $g=ad$. Thus $g\in AB$ and $g\in AC$. 
	
	%If $g\in A(B\cup C)$ then there are $a\in A$ and $d\in B\cup C$ with $g=ad$. If $d\in B$, then $g=ad\in AB$ and otherwise $g\in AC$. 
	
%	The proofs of (a) and (c) are straight forward. Furthermore the intersection formula in (a) implies (b). 
	%To show (b) note that (a) implies
%	\begin{align*}
%		L\partial_K(A)=&L(K\overline{A}\cap K \overline{A^c})\\
%			\subseteq &((L^{-1}K^{-1})A)\cap ((L^{-1}K^{-1})A^c)!!!!!\\
%			=&((KL)^{-1}A)\cap ((KL)^{-1}A^c)=\partial_{KL}(A)!!!!
%	\end{align*}
%\end{proof}

A net $(A_i)_{i\in I}$ of measurable subsets of $G$ is called \emph{finally somewhere dense}, if there is $j\in I$ such that for all $i\geq j$ the set $A_i$ is somewhere dense\footnote{$A\subseteq G$ is called \emph{somewhere dense}, if it has nonempty interior. Note that this ensures $\mu(A)>0$.}. 
%Furthermore a net $(A_i)_{i\in I}$ of subsets of $G$ we call \emph{point absorbing}, whenever for every $g\in G$ there is some $j\in I$ such that for all $i\geq j$ there is $g\in A_i$. 
A finally somewhere dense %and point absorbing 
net $(A_i)_{i\in I}$ of compact subsets of $G$ is called a \emph{Van Hove net}, if for all compact subsets $K\subseteq G$, there holds 
\begin{equation}\label{equ:VanHovecondition}
\lim_{i\in I}\frac{\mu(\partial_K A_i)}{\mu(A_i)}=0.
\end{equation}
A unimodular group is called \emph{amenable} whenever it contains a Van Hove net. 

\begin{remark}
	\begin{itemize}
		\item[(i)] A finally somewhere dense net $(A_i)_{i\in I}$ of compact sets is a Van Hove net iff for all symmetric compact sets $K$ the Van Hove condition (\ref{equ:VanHovecondition}) is satisfied. Indeed, if $K$ is an arbitrary non empty but compact set we can choose $k\in K$ and obtain $\partial_{K} A_i \subseteq \partial_{kK^{-1}K}A_i=k(\partial_{K^{-1}K}A_i)$. As $\mu(\partial_K A_i)\leq\mu(k(\partial_{K^{-1}K}A_i))=\mu(\partial_{K^{-1}K}A_i)$ holds and $K^{-1}K$ is symmetric and contains $e_G$ we obtain the claim. 
		\item[(ii)] A straight forward computation shows 
$ \partial_K A =\overline{KA}\setminus \left(\operatorname{int}\left(\bigcap_{k\in K} kA\right)\right) $
 		for $K\subseteq G$ compact and $A\subseteq G$.
		If we assume in addition $e_G\in K=K^{-1}$, then another computation gives $\partial_K A =\left((K\overline{A})\setminus\operatorname{int}(A)\right)\cup \left((K^{-1}\overline{A^c})\setminus \operatorname{int}(A^c)\right).$ 
  Thus the definitions of the $K$-boundary given above; in \cite{Tempelman}\footnote{Note that in \cite{Tempelman} the order of multiplication is inverse to our notation.}; in \cite{Schlottmann} and in \cite{fuhrmann2018irregular}, coincide, whenever $e_G\in K=K^{-1}$. As the discussed terms are monotone in $K$, we can adapt the arguments given in (i) to see that all definitions of $K$-boundary yield equivalent definitions of Van Hove nets. 
  \item[(iii)] We call a finally somewhere dense net $(A_i)_{i\in I}$ a F\o lner net, if for every $g\in G$ there holds $\lim_{i\in I}{\mu(g A_i\Delta A_i)}{\mu(A_i)}^{-1}=0,$
where $A\Delta B:=(A\setminus B)\cup (B\setminus A)$ is the \emph{symmetric difference} of  $A,B\subseteq G$. F\o lner nets are called ''left ergodic nets'' in \cite{Tempelman}.
%As by Proposition \ref{pro:Kboundary} our definition of $K$-boundary agrees with the definition given in \cite{Tempelman}, 
%We obtain the link between F\o lner nets and Van Hove nets from \cite[Appendix; (3.K)]{Tempelman} as follows. 
%\begin{center}%\label{pro:FolnerandVH}
	A net $(A_i)_{i\in I}$ is a Van Hove net, iff it is a F\o lner net and satisfies 
	$\lim_{i\in I} {\mu(\partial_W A_i)}{\mu(A_i)}^{-1}=0$
	for some %open 
	neighbourhood  $W$ of $e_G$ as presented in \cite[Appendix; (3.K)]{Tempelman}.
%\end{center}
	From this we obtain that every Van Hove net is a F\o lner net and that the notions of Van Hove and F\o lner nets are equivalent for discrete groups. Note that our definition of $K$-boundary and of Van Hove nets is inspired from \cite{Krieger}, where it is used to define F\o lner nets in discrete amenable groups. In \cite[Appendix; Example 3.4]{Tempelman} a F\o lner net in $\mathbb{R}^d$ is presented, that is not a Van Hove net. 
	
\item[(iv)] It is shown for $\sigma$-compact locally compact groups in \cite[Appendix 3.L]{Tempelman} and for second countable unimodular groups in \cite[Lemma 2.7]{pogorzelski2016banach} that the existence of Van Hove sequences is equivalent to the existence of F\o lner sequences. The corresponding arguments generalize to unimodular groups without countability assumptions if we consider nets instead of sequences. For further equivalent notions of amenability we recommend the monographs \cite{Pier,paterson2000amenability}. 
\end{itemize}
\end{remark}

\begin{proposition}\label{pro:VanHovenets+K}%\label
		Let $K,C\subseteq G$ be compact sets and $(A_i)_{i\in I}$ be a Van Hove net in $G$. Then $(K A_i)_{i\in I}$ and $(CA_i)_{i\in I}$ are Van Hove nets and satisfy $\lim_{i\in I}\frac{\mu(KA_i)}{\mu(CA_i)}=1$. 
	\end{proposition}
	
	\begin{proof} It clearly suffices to consider the case $C=\{e_G\}$. Let $L\subseteq G$ be compact. %Note that for every $K'\subseteq G$ compact and every $i\in I$ there holds $\overline{(KA_i)^c}\subseteq \overline{(kA_i)^c}=k\overline{A_i^c}\subseteq K\overline{A_i^c}$ for any $k\in K$, hence
		%\[\partial_{K'} (KA_i)=K'\overline{KA_i}\cap K'\overline{(KA_i)^c}\subseteq K'K\overline{A_i}\cap K'K\overline{A_i^c}=\partial_{K'K}A_i.\]
	As $LK$ is compact, we obtain $(KA_i)_{i\in I}$ to be a Van Hove net from
		\[0\leq \frac{\mu(\partial_{L} KA_i)}{\mu(KA_i)}\leq \frac{\mu(\partial_{LK}A_i)}{\mu(A_i)}\overset{i\in I}{\rightarrow} 0.\]
	%This proofs $(KA_i)_{i\in I}$ to be a Van Hove net. 	
	To show $\lim_{i\in I}\frac{\mu(KA_i)}{\mu(A_i)}=1$ let $k\in K^{-1}$ and note that %$e_g\in kK$. 
	% for every $i\in I$ there holds
	%\[kKA_i =A_i \cup (kKA_i\cap A_i^c)
	%\subseteq A_i \cup (kK\overline{A_i}\cap kK\overline{A_i^c})=A_i\cup \partial_{kK}A_i.\]
	%Hence %by Lemma \ref{lem:VHNbasicrules}(ii)
	$kKA_i\subseteq A_i\cup \partial_{kK} A_i$ implies 
	\[1\leq \frac{\mu(KA_i)}{\mu(A_i)}=\frac{\mu(kKA_i)}{\mu(A_i)}\leq 1+\frac{\mu(\partial_{kK}A_i)}{\mu(A_i)}\overset{i\in I}{\rightarrow}1.\]
%	Similarly one obtains $(CA_i)_{i\in I}$ to be a Van Hove net and $\lim \frac{\mu(A_i)}{\mu(CA_i)}=1$, hence
%	\[\lim_{i \in I}\frac{\mu(KA_{i})}{\mu(CA_{i})} =\left(\lim_{j\in I}\frac{\mu(KA_{j})}{\mu(A_{j})}\right)\left(\lim_{i\in I}\frac{\mu(A_{i})}{\mu(CA_{i})} \right)=1.\]
%	
	\end{proof}

%Lemma {lem:rescaleingVHN}!!!

\subsection{Uniform lattices in locally compact topological groups}\label{sub:Uniformlattices}
Let $G$ be a locally compact topological group. A discrete subgroup $\Lambda\subseteq G$ is called a \emph{uniform lattice}, whenever it is \emph{co-compact}, i.e.\ whenever $G\big/\Lambda$ is compact. 
A \emph{fundamental domain} of $\Lambda$ is a subset $C\subseteq G$ such that each element of $G$ can be written in a unique way as $g=cz$ with $c\in C$ and $z \in \Lambda$. 

%A discrete subgroup $\Lambda\subseteq G$ is called a \emph{uniform lattice}, if there is a pre-compact
%and Borel measurable $C$ that contains $e_G$ and satisfies $0<\mu(C)$ such that each $g\in G$ can be written uniquely as $g=cz$ with $c\in C$ and $z\in \Lambda$. The set $C$ is called a \emph{fundamental domain} for $\Lambda$ and satisfies $0<\mu(C)\leq \mu(\overline{C})<\infty$. 

\begin{remark} \label{remark:measures of Lambda and G}
	\begin{itemize}
	\item[(i)] Every uniform lattice allows the choice of a pre-compact\footnote{A subset $A$ of a topological space $X$ is called \emph{pre-compact}, whenever the closure $\overline{A}$ is compact in $X$.} 
 and Borel measurable fundamental domain $C$ such that $e_G\in C$ and such that $0<\mu(C)\leq \mu(\overline{C})<\infty$. Whenever we consider a fundamental domain we will assume such a choice. To see that such a choice is possible consider a pre-compact and open neighbourhood of $e_G$ such that $U^{-1}U\cap \Lambda=\{e_G\}$ and a compact set $K\subseteq G$ such that $K\Lambda=G$. Without lost of generality we assume $e_G\in K$. Let furthermore $(k_n)_{n=1}^N$ be a finite sequence in $K$ such that $k_1=e_G$ and such that $\bigcup_{n=1}^N (k_nU)\supseteq K$. Then 
 \[C:=\bigcup_{n=1}^N \left[(k_nU)\setminus\left(\bigcup_{i<n}k_iU\Lambda\right)\right]\]
satisfies the considered properties. %Note that we do not need countability assumptions for these arguments. 
%If $G$ is a  $\sigma$-compact locally compact abelian group, then also the reverse holds true and every cocompact discrete subgroup is a uniform lattice. Indeed, in order to construct a Fundamental domain $F$ for $\Lambda$, let $K\subseteq G$ be compact such that $K+\Lambda=G$ and choose a precompact and open neighbourhood $U$ of $0$ such that $(U+U)\cap \Lambda=\{0\}$. As $G$ is $\sigma$-compact there is a sequence $(g_n)_{n\in \mathbb{N}}$ in G such that $G=\bigcup_{n\in \mathbb{N}} g_n+U$. Choose sequences $(k_n)_{n\in \mathbb{N}}$ and $(k_n)_{n\in \mathbb{N}}$ in $K$ and $\Lambda$ respectively such that $g_n=k_n+l_n$. Now a straight forward argument shows  \[F:=\bigcup_{n\in \mathbb{N}} \left( (k_n+U)\setminus\bigcup_{m<n}(k_m+U+\Lambda)\right)\]
%to be the required fundamental domain. 

	\item[(ii)] % Our definition of uniform lattices implies $\Lambda$ to be \emph{cocompact}, i.e.\ the quotient $G\big / \Lambda$ to be compact. Thus 
	By \cite[Theorem 9.1.6]{deitmar2014principles} every locally compact group that contains a uniform lattice is unimodular.
If $F\subseteq \Lambda$ is finite, then the finite union $\bigcup_{z\in F} Cz$ is disjoint and measurable. Thus by the right invariance of the Haar measure there holds \[\mu(CF)=\sum_{z\in F} \mu(Cz)=\mu(C)|F|.\] 

\end{itemize}

\end{remark}

\begin{example} \label{exa:ExistenceofLattices}
	\begin{itemize}
		%\item[(i)] Every compact group $G$ contains the countable uniform lattice $\{e_G\}$ with fundamental domain $G$. 
		%\item[(ii)] The Euklidean space $\mathbb{R}^d$ contains the countable uniform lattice $\mathbb{Z}^d$ with fundamental domain $[0,1)^d$. 
		\item[(i)] If $H$ is a compact abelian group and $a,b\in \mathbb{N}$, then $\mathbb{R}^a\times \mathbb{Z}^b\times H$ contains the countable uniform lattice $\mathbb{Z}^{a+b}\times \{e_H\}$ with fundamental domain $[0,1)^a\times \{0\}^b\times H$. Note that up to isomorphism these are all compactly generated  locally compact abelian groups $G$. \cite[Theorem 9.8]{HewittandRoss}.
		\item[(ii)] The \emph{Heisenberg group} 
			$H_3(\mathbb{R}):=\left\{\begin{pmatrix}
1 & a & c \\
0 & 1 & b \\
0 & 0 & 1 
\end{pmatrix} ;\, a,b,c\in \mathbb{R} \right\}$
under matrix multiplication is a non-abelian amenable group and contains the uniform lattice $H_3(\mathbb{Z})$ with fundamental domain $H_3([0,1))$. For reference see \cite[Example 2.13]{eisner2015operator} and \cite[Exercise 1.2.4]{runde2004lectures}.
	\end{itemize}
\end{example}

\section{The Ornstein-Weiss lemma}\label{sec:Ornstein-Weiss groups}

	In the introduction we defined what it means that a group satisfies the Ornstein-Weiss lemma.
	%For abbreviation we call groups that satisfy the Ornstein-Weiss lemma \emph{Ornstein-Weiss groups}.  	
	From \cite[Theorem 1.1.]{Krieger} or \cite[Theorem 1.1]{FeketesLemma} we know that every discrete amenable group satisfies the Ornstein-Weiss lemma. Thus whenever a group $G$ contains a uniform lattice $\Lambda$, we know that the Ornstein-Weiss lemma holds in $\Lambda$. After a short remark on the history of the Ornstein-Weiss lemma we show that this implies that the Ornstein-Weiss lemma holds for $G$ as well, i.e. Theorem \ref{theint:OrnsteinWeisslemma}.

	\begin{remark}\label{rem:Ornsteinweissgroups}
 For the origins of the ideas of a proof of the Ornstein-Weiss lemma in countable amenable groups see \cite{OrnsteinandWeiss,ward1992abramov} and in particular \cite[1.3.1]{gromov1999topological}. These ideas are furthermore considered in \cite{HuangandYeandZhang,lindenstrauss2000mean,downarowicz2019tilings}. The corresponding arguments are worked out in detail for discrete amenable groups and even for discrete amenable semigroups in \cite{Krieger,FeketesLemma,Yan}. In the last part of the arguments presented in \cite{gromov1999topological} one uses that $\sup_{D\in \mathcal{K}(G)}f(D)/\mu(D)<\infty$ and in particular that this boundedness holds for $D\in \mathcal{K}(G)$ with small Haar measure, which are ''spread out'' a lot over $G$. See \cite{Krieger} for detail. In the discrete case $f(M)\leq |M|f(\{e_G\})$ follows easily from the right invariance and the subadditivity. Note that adding this relative boundedness to the assumptions on $f$ it is shown in \cite{pogorzelski2016banach} that a modified version of the Ornstein-Weiss lemma remains valid also for the groups considered in our context. However this boundedness assumption is not satisfied for the functions which are considered in the definition of topological entropy as we discuss in Remark \ref{rem:boundednessnotgiven}(iv) below. 
	%This gap in the theory is the main motivation to include a proof of the Ornstein Weiss lemma suitable for our purposes into this paper.  	
	%	We will see in Remark \ref{rem:topologicalconjugacy}(iv) that this assumption is already to strong for actions of $\mathbb{R}$ and we can not use this version of the Ornstein Weiss lemma obtained in \cite{pogorzelski2016banach}. As we are mainly interested in a version for $\mathbb{R}^d$ we will use that the Ornstein Weiss lemma holds for the uniform lattice $\mathbb{Z}^d$ and extrapolate from this a version of this lemma for $\mathbb{R}^d$. In order to give these arguments we will add  monotonicity to the assumptions on the considered set valued functions, which is trivially fulfilled in the context of defining entropy. These arguments also yield that the averages along Van Hove nets in $\mathbb{R}^d$ can be seen already through averages along Van Hove nets in $\mathbb{Z}^d$. This will give us in particular the chance of recycling well known results on topological entropy for discrete groups which we will carefully do. As these arguments generalize to locally compact amenable unimodular (not necessarily commutative) groups that contain a uniform lattice we will present them in this generality in Section  \ref{sec:Ornstein-Weiss groups}. Using the lattice structure of a cut and project scheme we will furthermore see that for all $\sigma$-compact locally compact abelian groups, which allow cut and project schemes over $G$, the Ornstein Weiss lemma remains valid.  
%\end{itemize}
	\end{remark}
	
		\subsection{Extrapolation from a uniform lattice}
	% This includes the Heisenberg-group, the additive group of $p$-adic numbers and all compactly generated locally compact abelian groups. 
	
	We will first construct Van Hove nets in a uniform lattice $\Lambda\subseteq G$ from Van Hove nets in $G$ with properties that allow to extract the validity of the Ornstein-Weiss lemma from the lattice. 
 
 \begin{lemma} \label{lem:constructingFhatandFcheck}\label{pro:VanHovenetsinLambda}
	Let $G$ be an amenable group and $\Lambda$ be a uniform lattice in $G$ with fundamental domain $C$. Then for every Van Hove net $(A_i)_{i\in I}$ in $G$ there exist Van Hove nets $(\check{F}_i)_{i \in I}$ and $(\hat{F}_i)_{i\in I}$ in $\Lambda$ that satisfy
\begin{itemize}
	\item[(i)] $C\check{F}_i\subseteq A_i\subseteq C\hat{F}_i$ for all $i\in I$ and
	\item[(ii)] $\lim_{i\in I}\frac{|\hat{F}_i|}{|\check{F}_i|}=1.$
\end{itemize} 	
 \end{lemma} 
 
 \begin{proof}
 	For $i\in I$ let $(A_i)_{i\in I}$ be a Van Hove net in $G$. Set  
 	$\check{F}_i:=\{z\in \Lambda;\, Cz\subseteq A_i\}$ %=\Lambda\setminus \left(C^{-1}A_i^c\right)$
 	and 
 	$\hat{F}_i:=\{z\in \Lambda;\, Cz\cap A_i\neq \emptyset\}.$ %=\Lambda\cap C^{-1}A_i.$
 	Then (i) follows directly from these definitions. In order to simplify the notation let $\check{A}_i:=C\check{F}_i$ and $\hat{A}_i:=C\hat{F}_i$ for $i\in I$. Let furthermore $K:=\overline{CC^{-1}}$ and note that $e_G \in K=K^{-1}$. The complements and boundaries in this proof are taken with respect to $G$ unless  otherwise mentioned. 
	For $z\in \Lambda$ we know that $z\notin C^{-1}A_i^c$ is equivalent to $Cz\cap A_i^c= \emptyset$ and thus obtain $\check{F}_i=\Lambda\setminus \left(C^{-1}A_i^c\right)$. A similar argument yields $\hat{F}_i=\Lambda\cap C^{-1}A_i.$

 	Before we show that $(\check{F}_i)_{i\in I}$ and $(\hat{F}_i)_{i\in I}$ are  Van Hove nets in $\Lambda$ we will show that the ratio of their cardinalities tends to $1$, i.e. (ii). We compute $\hat{A}_i=C(\Lambda\cap C^{-1}A_i)\subseteq C\Lambda\cap KA_i=KA_i$
	and similarly
	$(\check{A_i})^c=C(\Lambda\setminus\check{F}_i)=C(\Lambda\cap C^{-1}A_i^c)\subseteq% C\Lambda\cap KA_i^c=
	KA_i^c,$ from which we obtain $\overline{\hat{A}_i}\cap\overline{\check{A}_i^c}\subseteq K\overline{A_i}\cap K\overline{A_i^c}=\partial_K A_i$. Thus for all $i\in I$ there holds 
\begin{equation}\label{equ:blablabla}
 \overline{\hat{A}_i}\subseteq \partial_{K}A_i\cup \check{A}_i \text{ and } \overline{\check{A}_i^c}\subseteq \partial_{K}A_i\cup \hat{A}_i^c.
\end{equation}
We will need both inclusions later in this proof, but for now obtain from the first one that 
	$A_i\subseteq {\hat{A}_i}\subseteq \check{A}_i \cup \partial_K A_i.$
	Hence $\mu(\hat{A}_i)\leq \mu(\check{A}_i)+\mu(\partial_{K}A_i)$ and $\mu(A_i)-\mu(\partial_{K}A_i)\leq \mu(\check{A}_i)$. As $\lim_{i\in I}\frac{\mu(A_i)} {\mu(\partial_{K}A_i)}=\infty$ we get (ii) from the computation
	\begin{align*}
		1\leq \frac{|\hat{F}_i|}{|\check{F}_i|}&= \frac{\mu(\hat{A}_i)}{\mu(\check{A}_i)}		
		\leq \frac{\mu(\check{A}_i)+\mu(\partial_{K}A_i)}{\mu(\check{A_i})}
		=1+\frac{\mu(\partial_{K}A_i)}{\mu(\check{A_i})}\\
		&\leq 1+\frac{\mu(\partial_{K}A_i)}{\mu(A_i)-\mu(\partial_{K}A_i)}
		=1+\frac{1}{\frac{\mu(A_i)}{\mu(\partial_{K}A_i)}-1}.
	\end{align*}
We will finish the proof by showing that $(\check{F}_i)_{i\in I}$ and $(\hat{F}_i)_{i\in I}$ are Van Hove nets. To do this consider a finite set $F\in \Lambda$ and set $L:=\overline{CF}$. From (ii) we know the existence of $j\in I$ such that for all $i\geq j$ there holds
	\[1\leq \frac{\mu(\hat{A}_i)}{\mu(\check{A}_i)}=\frac{|\hat{F}_i|}{|\check{F}_i|}\leq 2.\]
	Hence $\mu(\check{A}_i)\leq \mu(A_i)\leq \mu(\hat{A_i})\leq 2 \mu(\check{A}_i)$.
 From (\ref{equ:blablabla}) and $\check{A}_i\subseteq \overline{A_i}$ we obtain furthermore 
	\[L\overline{\hat{A_i}}\subseteq L(\partial_K A_i\cup \overline{A_i})
		\subseteq L\partial_K A_i\cup L\overline{A_i}
		\subseteq \partial_{LK}A_i\cup L\overline{ A_i}\]
and analogously $ L\overline{\check{A_i}^c}\subseteq \partial_{LK} A_i\cup L \overline{A_i^c}.$ 
	 As there also holds 
$
	L\overline{\hat{A_i}^c}\subseteq L \overline{A_i^c} \subseteq \partial_{LK}A_i\cup L \overline{A_i^c}
$
and similarly $L\overline{\check{A_i}}\subseteq \partial_{LK}A_i\cup L \overline{A_i}$ 
%\begin{align*}
%	L\overline{\check{A_i}}\subseteq L \overline{A_i} \subseteq (\partial_{LK}(A_i))\cup( L \overline{A_i}).
%\end{align*}
we obtain from $L\subseteq LK$ that
\begin{align*}
	\partial_L\hat{A_i}\cup \partial_L\check{A_i}&=\left(L  \overline{\hat{A_i}} \cap L \overline{\hat{A_i}^c} \right) \cup \left(L  \overline{\check{A_i}} \cap L \overline{\check{A_i}^c} \right)\\
		&\subseteq \left(\partial_{LK}A_i\cup L \overline{A_i}\right) \cap \left(\partial_{LK}A_i\cup L \overline{A_i^c}\right) \\
		&=\partial_{LK}A_i\cup \left(L \overline{A_i} \cap  L \overline{A_i^c}\right) 
		%&=\partial_{LK}A_i\cup \partial_{L}A_i
		=\partial_{LK}A_i.
\end{align*}
 We thus obtain for all $i\geq j$ that
	 \begin{equation}\label{equ:blabla2}
	 	\frac{\mu(\partial_L \hat{A}_i)}{\mu(\hat{A}_i)}\leq \frac{\mu(\partial_{LK}A_i)}{\mu(A_i)} \text{ and }\frac{\mu(\partial_L \check{A}_i)}{\mu(\check{A}_i)}\leq 2\frac{\mu(\partial_{LK}A_i)}{\mu(A_i)}.
	 \end{equation}
	 Denoting the $F$-boundary taken of a subset $E$ with respect to $\Lambda$ or $G$ by $\partial_F^\Lambda E$ or $\partial_F^G E$, respectively, we use that $C$ is a fundamental domain to get 
\begin{align*}
	\partial^\Lambda_F \check{F}_i=& F\check{F}_i\cap F(\Lambda\setminus \check{F}_i)
	\subseteq FC\check{F}_i\cap FC(\Lambda\setminus \check{F}_i)\\
	=&FC\check{F}_i\cap F(G\setminus C\check{F}_i)
	\subseteq  F\overline{C\check{F}_i}\cap F\overline{G\setminus C\check{F}_i}
	= \partial^G_F \check{A}_i,
\end{align*}
hence $C \partial^\Lambda_{F}(\check{F}_i)\subseteq C \partial^G_{F}\check{A}_i\subseteq \partial^G_{CF}\check{A}_i
 	\subseteq \partial^G_{L}\check{A}_i.$ %Similarly we obtain  $C\partial^\Lambda_F\hat{F}_i\subseteq \partial^G_L\hat{A}_i$. 
 Thus, considering (\ref{equ:blabla2}) we compute
 	\begin{align*}
 		0\leq\frac{|\partial_F^\Lambda \check{F}_i|}{|\check{F}_i|}
 		%=\frac{|\partial_F^\Lambda \check{F}_n|\mu(C)}{|\check{F}_n|\mu(C)}
 		=\frac{\mu(C \partial_F^\Lambda \check{F}_i)}{\mu(\check{A}_i)} 
 		\leq 
 		%\frac{\mu(\partial_{L}^G (C\check{F}_i))}{\mu(C\check{F}_i)}=
 		\frac{\mu(\partial_{L}^G \check{A}_i)}{\mu(\check{A}_i)}
 		\leq 2\frac{\mu(\partial_{LK}^G {A}_i)}{\mu({A}_i)}.
 	\end{align*}
 As $LK$ is compact and $(A_i)_{i\in I}$ is a Van Hove net in $G$ we obtain $(\check{F}_{i})_{i\in I}$ to be a Van Hove net in $\Lambda$. Similarly one shows $(\hat{F}_i)_{i \in I}$ to be a Van Hove net.% 
 \end{proof}

%	The following shows all groups in Example \ref{exa:ExistenceofLattices} to be Ornstein-Weiss groups. 

	From the next theorem we obtain that every amenable group that contains a uniform lattice satisfies the Ornstein-Weiss lemma, i.e.\ the statement of Theorem \ref{theint:OrnsteinWeisslemma}. 
	
	\begin{theorem}\label{the:OrnsteinWeisslemma}%\label{rem:restrictionof_h_tolattice_OW_Formula}
	If $f\colon \mathcal{K}(G)\to \mathbb{R}$ is a subadditive, right invariant and monotone function and $(A_i)_{i\in I}$ is a Van Hove net in $G$, then  
	\[\lim_{i\in I}\frac{f(A_i)}{\mu(A_i)}=\frac{1}{\mu(C)}\lim_{j\in J}\frac{f(\overline{C} F_j)}{|F_j|}\]
	 holds for any Van Hove net $(F_j)_{j\in J}$ in a uniform lattice $\Lambda\subseteq G$ with fundamental domain \nolinebreak  $C$. 
	\end{theorem} 
 
\begin{proof} 
Let $C$ be a fundamental domain of $\Lambda$ in $G$ and note that $\mathcal{K}(\Lambda)$ is the set of finite subsets of $\Lambda$. In order to use that every discrete amenable group satisfies the Ornstein-Weiss lemma, we define
	\[f^\Lambda\colon \mathcal{K}(\Lambda) \to \mathbb{R}; F\mapsto f\left(\overline{C}F\right).\]
	It is straight forward to see, that $f^\Lambda$ is right invariant and monotone. In order to show, that $f^\Lambda$ is subadditive let $F,F'\in \mathcal{K}(\Lambda)$. As $\overline{C}(F\cup F')\subseteq \overline{C}F \cup \overline{C}F'$ we obtain from the monotonicity and the subadditivity of $f$ that 
	\[f^\Lambda(F\cup F')\leq f(\overline{C}F \cup \overline{C}F')\leq f^\Lambda(F)+f^\Lambda(F').\] 
	
	Let now $(A_i)_{i\in I}$ be a Van Hove net in $G$ and $(F_j)_{j\in J}$ be a Van Hove net in $\Lambda$. By Lemma \ref{lem:constructingFhatandFcheck} there are Van Hove nets $(\check{F}_i)_{i\in I}$ and $(\hat{F}_i)_{i\in I}$ such that 
	$C\check{F}_i\subseteq   A_i\subseteq  C\hat{F}_i$
	 for all $i\in I$ and $\lim_{i\in I}{|\hat{F}_i|}{|\check{F}_i|}^{-1}=1$. As $A_i$ is closed, we get furthermore $\overline{C}\check{F}_i\subseteq   A_i\subseteq  \overline{C}\hat{F}_i$ and hence 
	 \[f^\Lambda(\check{F}_i)\leq f(A_i)\leq f^\Lambda(\hat{F}_i).\]
	  Note that $\Lambda$ is a discrete amenable group by Lemma \ref{lem:constructingFhatandFcheck} and thus satisfies the Ornstein-Weiss lemma. This implies the existence of the following limits and 
	  \begin{align}\label{equ:Lambdalimits}
	  \lim_{i\in I}\frac{f^\Lambda(\check{F}_i)}{|\check{F}_i|}=\lim_{i\in I}\frac{f^\Lambda(\hat{F}_i)}{|\hat{F}_i|}=\lim_{j\in J}\frac{f^\Lambda({F}_j)}{|{F}_j|}.%=\lim_{j\in J}\frac{f(\overline{C}F_j)}{|F_j|}.
	  \end{align}
	Let $\epsilon>0$. As 
	$\lim_{i\in I}{|\hat{F}_i|}{|\check{F}_i|}^{-1}=1$
	and $|\check{F}_i|\leq |\hat{F}_i|$ for all $i\in I$ 
	 there is $j\in I$, such that for all $i\geq j$ there holds $|\hat{F_i}|\leq (1+\epsilon)|\check{F}_i|$ and hence
	\[\frac{1}{(1+\epsilon)}|\hat{F}_i|\mu(C)\leq \mu(C)|\check{F}_i|=\mu(C\check{F}_i) \leq \mu(A_i)\leq \mu(C\hat{F}_i)=\mu(C)|\hat{F}_i|\leq(1+\epsilon)|\check{F}_i|\mu(C) .\]
	Thus for $j\geq i$ there holds
	\begin{align*}
		\frac{1}{(1+\epsilon)}\frac{f^\Lambda(\check{F}_i)}{|\check{F}_i|}%&=\frac{1}{(1+\epsilon)}\frac{f(C\check{F}_i)}{|\check{F}_i|}\\
		&\leq \mu(C) \frac{f(A_i)}{\mu(A_i)}%\\
		%&\leq (1+\epsilon)\frac{Cf(\hat{F_i})}{|\hat{F}_i|}\\
		\leq(1+\epsilon)\frac{f^\Lambda(\hat{F}_i)}{|\hat{F}_i|}.
	\end{align*} 
	We obtain for every $\epsilon>0$ that
	\begin{align*}
	\frac{1}{(1+\epsilon)}\lim_{i\in I}\frac{f^\Lambda(\check{F}_i)}{|\check{F}_i|} &\leq\mu(C)\liminf_{i\in I} \frac{f(A_i)}{\mu(A_i)}
	\leq \mu(C) \limsup_{i\in I}\frac{f(A_i)}{\mu(A_i)}
	\leq (1+\epsilon)\lim_{i\in I}\frac{f^\Lambda(\hat{F}_i)}{|\hat{F}_i|} .
	\end{align*}
	This shows that the limit $\mu(C)\lim_{i\in I}\frac{f(A_i)}{\mu(A_i)}$ exists and that it equals the limits in (\ref{equ:Lambdalimits}).
%	\[\lim_{i\in I}\frac{f^\Lambda(\check{F}_i)}{|\check{F}_i|}=\lim_{i\in I}\frac{f^\Lambda(\hat{F}_i)}{|\hat{F}_i|}=\lim_{j\in J}\frac{f^\Lambda({F}_j)}{|{F}_j|}\]
	%for every Van Hove net $({F}_j)_{j\in J}$. 
	In particular it does not depend on the choice of $(A_i)_{i\in I}$. 
\end{proof}

	If $(G,H,\Lambda)$ is a CPS we know that $G\times H$ contains a uniform lattice and in order to show that $G$ satisfies the Ornstein-Weiss lemma one could hope that the properties of a CPS also imply that $G$ contains a uniform lattice. The next example was already studied in \cite{meyer1972algebraic} and shows that there are CPS with a physical space that contains no uniform lattices. 

\begin{example}\label{exa:CPS}
	Consider the additive group of the $p$-adic numbers $\mathbb{Q}_p$. %Here $p$ is a prime number. 
	For reference on $p$-adic numbers see \cite[Example 2.10]{baake2016spectral} and \cite{gouvea1997p}. Denote by $\mathbb{Z}[p^{-1}]$ the smallest subring of $\mathbb{Q}_p$ (or $\mathbb{R}$) that contains $\mathbb{Z}$ and $p^{-1}$. Then $(\mathbb{Q}_p,\mathbb{R},\Lambda)$ with $\Lambda:=\{(x,x);\, x\in \mathbb{Z}[p^{-1}]\}$ is a cut and project scheme. Furthermore the only discrete subgroup of $\mathbb{Q}_p$ is $\{0\}$, which is not co-compact. Thus $\mathbb{Q}_p$ does not contain a uniform lattice. For reference on this example see \cite[Example 5.C.10(2)]{cornulier2014metric} and \cite[Chapter II.10]{meyer1972algebraic}. 
\end{example}

	Nevertheless we can obtain all physical spaces of CPS to satisfy the Ornstein-Weiss lemma from the following result, which we already stated in Theorem \ref{theint:ProductandOW}. 

\begin{theorem}\label{the:ProductandOW}
	If $G$ and $H$ are amenable groups such that $G\times H$ satisfies the Ornstein-Weiss lemma, then $G$ and $H$ satisfy the Ornstein-Weiss lemma. 
\end{theorem}

To prove this theorem we need the following. 

\begin{lemma}\label{lem:productVHnets}
		Let $G$ and $H$ be unimodular groups and assume that $(A_i)_{i\in I}$ and $(B_j)_{j\in J}$ are Van Hove nets in $G$ and $H$ respectively. Then $(A_i\times B_j)_{(i\times j)\in I \times J}$ is a Van Hove net in $G\times H$, where $I\times J$ is ordered component wise. 
	\end{lemma}
	\begin{proof}
		Consider a compact subset $M\subseteq G\times H$ and let $K:=\pi_G(M)$ and $C:=\pi_H(M)$ the projections. A straight forward computation shows 
		\[\partial_M (A_i\times B_j)\subseteq (\partial_K A_i \times \partial_C  B_j)\cup (KA_i\times \partial_C B_j)\cup(\partial_K A_i \times C B_j),\]
		for all $i\in I$ and $j\in J$. Thus Proposition \ref{pro:VanHovenets+K} yields
		\begin{align*}
		0&\leq \frac{\mu_G\times \mu_H(\partial_M(A_i\times B_j))}{\mu_G\times \mu_H(A_i\times B_j)}\\
		&\leq\frac{\mu_G(\partial_K A_i) \mu_H(\partial_C B_j)}{\mu_G(A_i) \mu_H(B_j)}+\frac{\mu_G(K A_i) \mu_H(\partial_C B_j)}{\mu_G(A_i) \mu_H(B_j)}+\frac{\mu_G(\partial_K A_i) \mu_H(C B_j)}{\mu_G(A_i) \mu_H(B_j)}\overset{(i,j)\in I\times J}{\longrightarrow} 0.
		\end{align*}
	\end{proof}
	
\begin{proof}[Proof of Theorem \ref{the:ProductandOW}]
	Clearly $G$ and $H$ are amenable groups as the projections of the amenable group $G\times H$. Consider a monotone right-invariant and subadditive mapping $f\colon \mathcal{K}(G)\to \mathbb{R}$ and a Van Hove net $(A_i)_{i\in I}$. Denote by $\mu_H$ the Haar measure of $H$ and choose the Haar measure $\mu_{G\times H}$ as the product measure $\mu_G\times \mu_H$. Let $(B_j)_{j\in J}$ be any Van Hove net in $H$. It is easy to see that $h\colon \mathcal{K}(G\times H)\to \mathbb{R}$ defined by 
\[h(Q):=\inf\left\{\sum_{n=1}^N f(C_n)\mu_H(D_n);\, N\in \mathbb{N}, C_n\subseteq G, D_n\subseteq H, Q\subseteq \bigcup_{n=1}^N C_n\times D_n\right\}\]
is monotone, right invariant and subadditive. We next show that for compact subsets $A\subseteq G$ and $B\subseteq H$ there holds $h(A\times B)=f(A)\mu_H(B)$. Clearly there holds $h(A\times B)\leq f(A)\mu_H(B)$. To show the other inequality let $C_n\subseteq G$ and $D_n\subseteq H$ such that $A\times B\subseteq \bigcup_{n=1}^N C_n\times D_n$. Without lost of generality we assume $D_n\subseteq B$ for $n\in \mathcal{N}:=\{1,\cdots,N\}$. Let furthermore $\{E_1,\cdots,E_M\}$ be a finite Borel partition of $B$ s.t.\ $\bigcup_{m\in \mathcal{M}_n}E_m=D_n$ for all $n\in \mathcal{N}$, where we denote $\mathcal{M}:=\{1,\cdots,M\}$ and $\mathcal{M}_n:=\{m'\in \mathcal{M};\, E_{m'}\subseteq D_n\}$. Setting $\mathcal{N}_m:=\{n'\in \mathcal{N};\, E_{m}\subseteq D_{n'}\}$ one obtains $A\subseteq \bigcup_{n\in \mathcal{N}_m} C_n$ and thus by the subadditivity of $f$ that $f(A)\leq \sum_{n\in \mathcal{N}_m}f(C_n)$ for all $m\in \mathcal{M}$. Hence
\begin{align*}
	\sum_{n\in \mathcal{N}}f(C_n)\mu_H(D_n)&=\sum_{n\in \mathcal{N}}\sum_{m\in \mathcal{M}_n} f(C_n)\mu_H(E_m)=\sum_{n\in \mathcal{N},m\in \mathcal{M},E_m\subseteq D_n} f(C_n)\mu_H(E_m)\\
		&=\sum_{m\in \mathcal{M}}\sum_{n\in \mathcal{N}_m} f(C_n)\mu_H(E_m)\geq \sum_{m\in \mathcal{M}} f(A)\mu_H(E_m)=f(A)\mu_H(B).
\end{align*}
This shows $h(A\times B)=f(A)\mu_H(B)$ for compact subsets $A\subseteq G$ and $B\subseteq H$. By Lemma \ref{lem:productVHnets} we obtain that $(A_i\times B_j)_{(i,j)\in I \times J}$ is a Van Hove net in $G\times H$. As $G\times H$ satisfies the Ornstein-Weiss lemma this implies that
	\[\frac{f(A_i)}{\mu_G(A_i)}=\frac{f(A_i)\mu_H(B_j)}{\mu_G(A_i)\mu_H(B_j)}=\frac{h(A_i\times B_j)}{\mu_{G\times H}(A_i\times B_j)},\] converges to a limit independent from $(A_i)_{i\in I}$ and we obtain that $G$ satisfies the Ornstein Weiss lemma. 
\end{proof}

%Now Theorem \ref{the:CPSimpliesOW} can be obtained by the following arguments. If $(G,H,\Lambda)$ is a cut and project scheme, then $G\times H$ contains a uniform lattice and is thus an Ornstein-Weiss group by Proposition \ref{pro:OrnsteinWeisslemma}. Thus $G$ and $H$ are Ornstein-Weiss groups by Theorem \ref{the:ProductandOW}. 

\section{Basics of entropy theory}\label{sec:Entropy theory for Ornstein-Weiss groups}

% Actions in the context of CPS are defined naturally as actions on compact Hausdorff uniform spaces \cite{Schlottmann,baake2004dynamical}. 
%  As the arguments for introducing relative topological entropy are well known we took the freedom to present them in the slightly more general language of uniformities. The treatment of uniformities also allows to give a common ground for the metric space approach presented in the introduction and the corresponding terminology presented in \cite{Tagi-Zade}. 
During this section let $G$ be an amenable group that satisfies the Ornstein-Weiss lemma. 
  % These approaches use certain metrical distances or finite open covers respectively for scaling, while our approach uses the concept of entourages. Details will be presented in section \ref{sec:Actions on metric and topological spaces} below. Another motivation is that parts of the theory like the invariance under conjugation can be seen directly from the definitions with our approach and that . 

%
% of Ornstein Weiss groups acting on compact uniform spaces.  the relative topological entropy for actions of Ornstein-Weiss groups on compact uniform spaces and relate this approach to the better known approach via the Bowen metric for actions on compact metric spaces \cite{BrinandStuck,bowen1971entropy} and the approach for compact metric spaces by Tagi-Zade \cite{Tagi-Zade} using finite open covers for scaling. 

\subsection{Bowen entourage}

For an action $\phi\colon G\times X \to X$ on a compact Hausdorff space, an entourage $\eta\in \mathbb{U}_X$ and a compact subset $A\subseteq G$ we define the \emph{Bowen entourage} as
\[\eta_A:=\left \{(x,y);\, \forall g\in {A} : (\phi^g(x),\phi^g(y))\in \eta\right\}=\bigcap_{g\in {A}}\left(\phi^g \times \phi^g\right)^{-1}(\eta).\] 
We will show in Lemma \ref{lem:Bowenentourageisentourage} below that the Bowen entourage are indeed entourages of $X$.
In order to omit brackets we will use the convention, that the operation of taking a Bowen entourage is a stronger operation than the product of entourages.

\begin{remark} \label{rem:Bowenmetricundentourages}
	%The definition of the Bowen action is inspired by the definition of Bowen metric $d_A$, as defined in \cite{bowen1971entropy} and presented in the introduction. 
	% for actions of $\mathbb{Z}$. Let $\phi\colon G\times X \to X$ be a flow on a metric space $(X,d)$ and for $A\subseteq G$ compact define the \emph{Bowen metric} by $d_A(x,y):=\max_{g\in {A}}d(\phi^g(x),\phi^g(y))$ for $x,y\in X$. 
	It is straight forward to show, that $d_A$ is a metric and that $[d_A<\epsilon]=[d<\epsilon]_A$ for all compact $A\subseteq G$ and $\epsilon>0$. %Thus a subset $M\subseteq X$ has diameter with respect to $d_A$ less than $\epsilon$ iff $M^2 \subseteq ([d<\epsilon])_A$. 
\end{remark}

%The continuity of $\phi$ and the compactness of $X$ imply, that the image of the Bowen action is indeed contained in $\mathbb{U}_X$.

\begin{lemma} \label{lem:Bowenentourageisentourage}
	Let $\phi \colon G\times X \to X$ be a flow on a compact Hausdorff space. 
	For every $\eta \in \mathbb{U}_X$ and every compact subset $A\subseteq G$  there holds $\eta_A\in \mathbb{U}_X$. 
\end{lemma}
\begin{proof}% As $\eta_A=\eta_{\overline{A}}$ we assume without lost of generality $A$ to be compact. 

%For a map $f\colon X \to Y$ we write $f\times f\colon X^2 \to Y^2$ for the direct product defined by $f\times f(x,y):=(f(x),f(y))$. A map $f\colon X \to Y$ between compact Hausdorff spaces $X$ and $Y$ is called \emph{uniformly continuous}, if the preimage of every entourage of $Y$ under $f\times f$ is an entourage of $X$. Note that since we assume $X$ to be compact we obtain that $f$ is continuous iff $f$ is uniformly continuous \cite[Theorem 6.31]{Kelley}. 

	Note that $\phi\colon A\times X \to X$ is a continuous mapping from a compact Hausdorff space. Thus by \cite[Theorem 6.31]{Kelley}
	%\[\{(g,x,g',x')\in (A\times X)\times (A\times X);\, (\phi(g,x),\phi(g',x'))\in \eta\}\]
	$(\phi\times \phi)^{-1}(\eta)$
	is contained in the uniformity of $A\times X$. For $\kappa \in \mathbb{U}_A\text{ and } \theta \in \mathbb{U}_X$ we set $\kappa \overline{\times} \theta :=\{(g,x,g',x')\in  (A\times X)\times (A\times X);\, (g,g')\in \kappa \text{ and } (x,x')\in \theta \}$. As
	$ \{\kappa \overline{\times} \theta ;\, \kappa \in \mathbb{U}_A\text{ and } \theta \in \mathbb{U}_X\} $ is a base for the product uniformity on $A\times X$ there are $\kappa\in \mathbb{U}_A$ and $\theta\in \mathbb{U}_X$ with 
	\begin{align*}
		\kappa \overline{\times} \theta \subseteq& (\phi\times \phi)^{-1}(\eta) 
	= \{(g,x,g',x')\in (A\times X)\times (A\times X);\, (\phi^g(x),\phi^{g'}(x'))\in \eta\}.
	\end{align*}
	For $(x,x')\in \theta$ and $g\in A$ there holds $(g,x,g,x')\in \kappa \overline{\times} \theta$ and we obtain $(\phi^g(x),\phi^g(y))\in \eta$. This proves $\theta \subseteq \eta_A$ and hence
	$\eta_A\in \mathbb{U}_X$. 
\end{proof}

%\begin{remark}\label{rem:dAinducesUX}
%	Let $\phi$ be an action of $G$ on a compact metric space. Note that the previous lemma can be seen as the natural generalization of the fact that all Bowen metrics with respect to $\phi$ are equivalent, as it follows from Lemma \ref{lem:Bowenentourageisentourage} and Remark \ref{rem:Bowenmetricundentourages} that they induce the same uniformity. 
%	
%	% i.e.\ they induce the same topology. Indeed, they induce the same uniformity. To see this let $A\subseteq G$ be compact. Observe that $[d_A<\epsilon]=[d<\epsilon]_A$ is contained in the uniformity generated by $d$ for $\epsilon>0$, as seen in the previous lemma. Furthermore $[d<\epsilon]\supseteq [d_A<\epsilon]$ is contained in the uniformity generated by $d_A$. Hence the uniformities of $d_A$ and of $d$ coincide. 
%\end{remark}

The following is straight forward to prove and justifies to write $\eta_{AB}$ for $\eta_{(AB)}=(\eta_A)_B$. % for all entourages $\eta\in \mathbb{U}_X$ and compact subsets $A,B\subseteq G$. 

\begin{proposition} \label{pro:Bowenentouragesarules}
	For $\eta,\kappa\in \mathbb{U}_X$ and compact subsets $A,B\subseteq G$  there holds 
	$\eta_{(AB)}=(\eta_A)_B$,
	$\eta_{A\cup B}=\eta_A\cap\eta_B$ and
	$\eta_A\kappa_A\subseteq (\eta \kappa)_A$.
%\begin{itemize}
%	\item[(i)] 	$\eta_{(AB)}=(\eta_A)_B$,
%	\item[(ii)] $\eta_{A\cup B}=\eta_A\cap\eta_B$ and
%	\item[(iii)] $\eta_A\kappa_A\subseteq (\eta \kappa)_A$.
%\end{itemize}	
\end{proposition}
%\begin{proof} Straight forward.
%To show (a) we calculate
%	\begin{align*}
%		(\eta_A)_B&=\bigcap_{b\in B} \left((\phi^b\times\phi^b)^{-1}\left(\bigcap_{a\in A} (\phi^a\times \phi^a)^{-1}(\eta)\right)\right)\\
%		&=\bigcap_{(a,b)\in A\times B} \left((\phi^b\times\phi^b)^{-1}\circ (\phi^a\times \phi^a)^{-1}(\eta)\right)\\
%		&=\bigcap_{(a,b)\in A\times B} \left((\phi^{ab}\times\phi^{ab})^{-1}(\eta)\right)\\
%		&=\eta_{(AB)}.
%	\end{align*}
%	To show (b) we calculate
%	\begin{align*}
%		\eta_A\cap\eta_B&=\left(\bigcap_{a\in A} (\phi^a\times\phi^a)^{-1}(\eta)\right)\cap\left(\bigcap_{b\in B} (\phi^b\times\phi^b)^{-1}(\eta)\right)\\
%		&=\left(\bigcap_{c\in A\cup B} (\phi^c\times\phi^c)^{-1}(\eta)\right)\\
%		&=\eta_{A\cup B}.
%	\end{align*}
% To show (c): 
%For $(x,y)\in \eta_A\eta_A$ there is $z\in X$ with $(x,z),(z,x)\in \eta_A$. Thus for all $g\in A$ there holds $(\phi^g(x),\phi^g(z)),(\phi^g(z),\phi^g(y))\in \eta$, hence $(\phi^g(x),\phi^g(y))\in \eta \eta$. This shows $(x,y)\in (\eta \eta)_A$.
%\end{proof}

%\begin{remark}\label{rem:dAB=dA_B}
%	Let $\phi$ be an action on a compact metric space. Proposition \ref{pro:Bowenentouragesarules}(a) and (b) can be interpreted as the generalization of the fact, that for $A,B\subseteq G$ compact and $x,y\in X$ there holds $d_{AB}(x,y)=(d_A)_B(x,y)$ and $d_{A\cup B}(x,y)=\max\{d_A(x,y),d_B(x,y)\}$.
%\end{remark}

\subsection{Relative topological entropy}

The following approach to relative topological entropy is inspired by the approach to topological entropy of $\mathbb{Z}$-actions on compact metric spaces via sets of small diameter, given in \cite[Section 2.5]{BrinandStuck}. Consider first a compact Hausdorff space $X$ and $\eta \in \mathbb{U}_X$.\begin{definition}
	For $\eta\in \mathbb{U}_X$ we say that a subset $M\subseteq X$ is \emph{$\eta$-small}, if any $x\in M$ is $\eta$-close to any $y\in M$, i.e.\ iff $M^2\subseteq \eta$.  We say, that a set $\mathcal{U}$ of subsets of $X$ is of \emph{scale $\eta$}, if $U$ is $\eta$-small for every $U\in \mathcal{U}$.
As $X$ is compact there is a finite open cover of $X$ of scale $\eta$. Thus for every $M\subseteq X$ there exists a finite open cover of scale $\eta$ as well. 
For $M\subseteq X$ and $\eta\in \mathbb{U}_X$ we denote by
	$\operatorname{cov}_M(\eta)$
the minimal cardinality of an open cover of $M$ of scale $\eta$. If $p\colon X\to Y$ is a map to some set $Y$, we define 
\[ \operatorname{cov}_p(\eta):=\sup_{y\in Y}\operatorname{cov}_{p^{-1}(y)}(\eta).\]
%
%\begin{remark}\label{rem:coverofscaleanddiameter}
%	Let $(X,d)$ be a metric space. The \emph{diameter} of a subset $M$ is defined as 
%	$\operatorname{diam}(M):=\sup_{(x,y)\in M^2}d(x,y).$	
%The \emph{diameter} of a set $\mathcal{U}$ of subsets of $X$ is 
%	$\operatorname{diam}(\mathcal{U})=\sup_{U\in \mathcal{U}} \operatorname{diam}(U)$.  For $\epsilon>0$ a subset $M\subseteq X$ satisfies $\operatorname{diam}(M)\leq \epsilon$ iff it is $[d\leq\epsilon]$-small, and an open cover satisfies $\operatorname{diam}(\mathcal{U})\leq \epsilon$ iff it is of scale $[d\leq\epsilon]$. 
%\end{remark}

\end{definition}

%$T_1$-spaces?

%It is immediate that $\eta \mapsto \operatorname{cov}_p(\eta)$ is decreasing. As the Bowen action is decreasing in the second argument, we obtain 
A well known argument shows that
$\mathcal{K}(G) \ni A\mapsto \log (\operatorname{cov}_p(\eta_{A}))$ is monotone, right invariant and subadditive for every $\eta\in \mathbb{U}_X$. 
%for every action $\pi$ of an Ornstein-Weiss group $G$ on a compact uniform space $X$ and every factor map $p$ from $\pi$. 
Thus the limit in the following definition of relative topological entropy exists and is independent from the choice of a Van Hove net.  

\begin{definition}
	Let $\phi$ be an action of $G$ on a compact Hausdorff space $X$ and $\psi$ be a factor of $\phi$ via factor map $p$. For some Van Hove net $(A_i)_{i\in I}$ and $\eta\in \mathbb{U}_X$, we define %the \emph{relative topological entropy on covering scale $\eta$} as
	\[\operatorname{E}(\eta|\phi \overset{p}{\to} \psi):=\lim_{i\in I} \frac{\log(\operatorname{cov}_p(\eta_{A_i}))}{\mu(A_i)}.\] 
	We furthermore define the \emph{relative topological entropy of $\phi$ relative to $\psi$} as
	\[\operatorname{E}(\phi \overset{p}{\to} \psi):=\sup_{\eta \in \mathbb{U}_X}\operatorname{E}(\eta|\phi \overset{p}{\to} \psi).\]
 %Furthermore, we define the \emph{entropy} $\operatorname{E}(\tau)$ of $\tau$, as the relative entropy of $\tau$ under the condition of the one point flow. 
	The \emph{topological entropy} of $\phi$ is defined as the relative topological entropy relative to the one point flow. Note that in this case $\operatorname{cov}_p(\eta)=\operatorname{cov}_X(\eta)$ is the minimal cardinality of an open covering of $X$ of scale $\eta$.
\end{definition}

\begin{remark}\label{rem:topologicalconjugacy} \label{rem:boundednessnotgiven}	
\begin{itemize}
	\item[(i)] There holds $\operatorname{E}(\phi \overset{p}{\to} \psi)=\sup_{\eta \in \mathbb{B}_X}\operatorname{E}(\eta|\phi \overset{p}{\to} \psi)$
		for any base $\mathbb{B}_X$ of $\mathbb{U}_X$, as $\operatorname{E}(\eta|\phi \overset{p}{\to} \psi)$ increases whenever $\eta$ decreases with respect to set inclusion. 
		\item[(ii)] If $(X,d)$ is a compact metric space, we can choose the canonical base $\{[d<\epsilon];\, \epsilon>0\}$ to obtain the definition given in the introduction. 
%	\item[(iii)] Each topological conjugacy  between compact uniform spaces defines a bijection between the uniformities of the corresponding phase spaces. Thus we obtain the well known fact that topological entropy is invariant under topological conjugacy. %This is well known for actions on compact metric spaces and shown for example in \cite[Corollary 2.5.4.]{BrinandStuck}.
%	\item[(iv)] Note that we do not need $p$ to be a continuous map. In fact we can define relative topological entropy of a (continuous) action $\phi$ with respect to a {set theoretic factor map}\footnote{
%	For a group $G$ and a set $Y$, we call a mapping $\psi\colon G\times Y\to Y$ a \emph{set theoretic action} of $G$ on $Y$, whenever $\psi$ is a (continuous) action after we equip $G$ and $Y$ with the discrete topology. 	
%	% the  identity on $X$ and for all $g,g'\in G$ there holds $\phi(g,\phi(g',\cdot))=\phi(gg',\cdot)$. 
%	If $\psi \colon G\times X\to X$ is a further set theoretic action we call a map $p\colon X \to Y$ a \emph{set theoretic factor map}, if it is a (continuous) factor map after we equip $G$, $X$ and $Y$ with the discrete topology.
%	%Furthermore if $\phi\colon G\times X\to X$ and $\psi \colon G\times Y\to Y$ are set theoretic actions, we call a map $p\colon X \to Y$ a \emph{set theoretic factor map}, if $p(\phi(g,x))=\psi(g,p(x))$ holds for all $(g,x)\in G\times X$.
%	}
%	$p\colon X\to Y$ to a set theoretic action $\psi\colon G\times Y\to Y$.
	
\item[(iii)] If $X$ is a compact Hausdorff space, then $\{\bigcup_{U\in \mathcal{U}}U^2;\, \mathcal{U}$ finite open cover of $X\}$ is a base of the uniformity of $X$. Using this base one obtains the definitions used in \cite{Tagi-Zade}. 
\item[(iv)] Consider the continuous rotation $\mathbb{R}\times \mathbb{T}\to \mathbb{T}\colon (g,x)\mapsto g+x\mod 1$ with $\mathbb{T}:=\mathbb{R}\big/\mathbb{Z}$. There holds $\log(\operatorname{cov}_{\mathbb{T}}[d<\delta])=\log(\operatorname{cov}_{\mathbb{T}}[d_{B_\epsilon}<\delta])$ for all $\epsilon,\delta >0$, where $B_\epsilon$ denotes the centred closed ball of radius $\epsilon$. We thus obtain that $A\mapsto {\log(\operatorname{cov}_{\mathbb{T}}([d<\delta]_A))}{\mu(A)}^{-1}$
is not bounded, whenever $\delta$ is chosen small enough. This observation is the reason why we can not assume the boundedness of $\mathcal{K}(G)\ni A\mapsto {f(A)}{(\mu(A))^{-1}}$ like in \cite{pogorzelski2016banach}. Compare  with Remark \ref{rem:Ornsteinweissgroups}.
\end{itemize}
\end{remark}

\subsection{Relative topological entropy via spanning and separating sets}

It is well known that one can also define topological entropy of $\mathbb{Z}$-actions on compact metric spaces in terms of separated and of spanning sets \cite{bowen1971entropy,BrinandStuck}. %See for example \cite[Section 2.5]{BrinandStuck}.
 In \cite{Hood} this approach is generalized to $\mathbb{Z}$-actions of compact Hausdorff spaces. As this approach is important in the context of aperiodic order \cite{baake2007pure,JaegerLenzandOertel,fuhrmann2018irregular} we give a brief recap. Consider a compact Hausdorff space $X$.

\begin{definition}
For $\eta\in \mathbb{U}_X$ a subset $S\subseteq X$ is called \emph{$\eta$-separated}, if for every $s\in S$ there is no further element in $S$ that is $\eta$-close to $s$. 
Furthermore we say that $S\subseteq X$ is \emph{$\eta$-spanning} for $M\subseteq X$, if for all $m\in M$ there is $s\in S$ such that $s$ is $\eta$-close to $m$ or $m$ is $\eta$-close to $s$. 
\end{definition}

\begin{remark}\label{rem:Bowenmetricandseparation}
%Both notions are symmetric, i.e. $S$ is $\eta$-separated ($\eta$-spanning for $M$) iff it is $\eta^{-1}$-separated ($\eta^{-1}$-spanning for $M$). 
	A subset $S$ of a metric space $(X,d)$ is $[d<\epsilon]$-separated, if any two distinct points in $S$ are at least $\epsilon$ apart, i.e. $d(x,y)\geq \epsilon$ for all $x,y\in S$ with $x\neq y$. Furthermore $S$ is $[d<\epsilon]$-spanning for $M\subseteq X$, iff for every $m\in M$ there is $s\in S$ such that $d(s,m)<\epsilon$. 
\end{remark}
%
%\begin{itemize}
%\item[(i)] For $\eta\in \mathbb{U}_X$ and $M\subseteq X$ the following statements are equivalent
%	\begin{itemize}
%	\item[(a)] $M$ is $\eta$ separated,
%	\item[(b)] $\eta[m]$ is disjoint from $M\setminus \{m\}$ for every $m\in M$,
%	\item[(c)] for any two distinct elements $x,y\in M$ we have that $x$ is not $\eta$-close to $y$.
%	\end{itemize}		
%\end{itemize}
%This holds, iff  and is equivalent to the condition that for every  and vice versa. 
%%If $\eta$ is symmetric, then $M$ is $\eta$-separated, iff every two distinct elements $x,y\in M$ are not $\eta$-close. 
% 
%
%
%\begin{remark}
%spanning:
%Note that this is equivalent to the condition that $\eta[k]$ is not disjoint from $M$, for every $k\in K$; and also equivalent to the condition that $\{[m]\eta;\,m\in M\}$ covers $K$. %, which holds iff $[M]\eta=X$. 
%Note 
%\end{remark}

With similar arguments as used in metric spaces we obtain the following lemma. 
\begin{lemma} \label{lem:SepandCov}\label{lem:SpaandSep} 
	For $\eta\in \mathbb{U}_X$ and $M\subseteq X$ the cardinality of every $\eta$-separated subset $S\subseteq M$ is bounded from above by $\operatorname{cov}_M(\eta)<\infty$. In particular there are finite $\eta$-separated subsets of $M$ of maximal cardinality. Every $\eta$-separated subset $S\subseteq M$ of maximal cardinality is $\eta$-spanning for $M$. In particular there are finite subsets of $M$ that are $\eta$-spanning for $M$.
\end{lemma}

\begin{definition}
	For $\eta\in \mathbb{U}_X$ and $M\subseteq X$ we define
	$\operatorname{sep}_M(\eta)$
	as the maximal cardinality of a subset of $M$ that is $\eta$-separated and
	$\operatorname{spa}_M(\eta)$
	as the minimal cardinality of a subset of $M$ that is $\eta$-spanning for $M$. 
	For a map $p\colon X\to Y$ to some set $Y$ we define
	\begin{center}
		$\operatorname{sep}_p(\eta):=\sup_{y\in Y} \operatorname{sep}_{p^{-1}(y)}(\eta)$ ~and ~$\operatorname{spa}_p(\eta):=\sup_{y\in Y} \operatorname{spa}_{p^{-1}(y)}(\eta).$
	\end{center}
%	
%	\[\operatorname{sep}_p(\eta):=\sup_{y\in Y} \operatorname{sep}_{p^{-1}(y)}(\eta)\]
%	and
%	\[\operatorname{spa}_p(\eta):=\sup_{y\in Y} \operatorname{spa}_{p^{-1}(y)}(\eta).\]
%	\begin{itemize}
%	%	\item[] $\operatorname{cov}_D(\eta):=\min\{|\mathcal{U}|;\, \mathcal{U}$ open cover of $D$ of scale $\eta\}$,
%		\item[] $\operatorname{sep}_D(\eta):=\max\{|M|;\, M\subseteq D$ is $\eta$-separated$\}$,
%		\item[] $\operatorname{spa}_D(\eta):=\min\{|M|;\, M\subseteq D$ is $\eta$-spanning for $D\}$. 
%\end{itemize}	 
\end{definition}

Unfortunately the Ornstein-Weiss lemma can not be applied directly to these notions. Nevertheless we use $\operatorname{cov}_p$ to show that $\operatorname{spa}_p$ and $\operatorname{sep}_p$ can be used to define entropy independently from the choice of a Van Hove net. A straight forward argument shows the following. 

%\begin{lemma}\label{lem:SpaSepCovStaticRelationI}
%	For symmetric entourages  $\eta,\theta \in \mathbb{U}_X$ and a subset $M\subseteq X$ there holds 
%	%\begin{itemize}
%	 		%\item[(i)] $\operatorname{spa}_{M}(\eta)\leq\operatorname{sep}_{M}(\eta)\leq \operatorname{cov}_{M}(\eta)$ and
%	 		%\item[(ii)] $\operatorname{spa}_{M}(\eta)\leq \operatorname{sep}_{M}(\eta)$ and
%	 		%\item[(ii)] 
%	 		$\operatorname{cov}_{M}(\theta\eta\eta\theta)\leq \operatorname{spa}_{M}(\eta)$, whenever $\theta$ is open. 
%% 	\end{itemize}
%\end{lemma}
%\begin{proof}
%	%From Lemma \ref{lem:SepandCov} and Lemma \ref{lem:SpaandSep} we obtain (i).  To show (ii) l
%	Let $S\subseteq M$ be $\eta$-spanning for $M$. As $\theta$ is open and $\eta$ is symmetric we obtain $\{\theta\eta[s];\, s \in S\}$ to be an open cover of $M$. A straight forward argument shows this cover to be of scale $\theta\eta\eta\theta$. 
%%	
%%	It suffices to show that $\theta\eta[s]$ is $(\theta\eta\eta\theta)$-small for any $s\in S$. For $x,y\in \theta\eta[s]=\theta^{-1}\eta^{-1}[s]=(\eta\theta)^{-1}[s]$ we know $x$ to be $\theta\eta$-close to $s$ and $s$ to be $\eta\theta$-close to $y$, hence $x$ to be $\theta\eta\eta\theta$-close to $y$.
%\end{proof}

%The following allows to link these notions to the definition of relative topological entropy. 

\begin{lemma} \label{lem:SpaSepCovNONStaticRelation}
	 Let $\eta\in \mathbb{U}_X$ and $p\colon X\to Y$ be a map to a set $Y$. Then there exists an  entourage $\theta\in \mathbb{U}_X$ with $\theta\subseteq \eta$ such that for every compact $A\subseteq G$ there holds 
	 \[\operatorname{cov}_{p}(\eta_A)\leq \operatorname{spa}_{p}(\theta_A)\leq \operatorname{sep}_{p}(\theta_A)\leq \operatorname{cov}_{p}(\theta_A).\]
	 \end{lemma}

\begin{theorem} \label{the:Entropyingenerality}
Let $\phi\colon G\times X \to X$ be an action of $G$, on a compact Hausdorff space $X$. % and $\mu$ a Haar Measure on $G$. 
Let furthermore $\psi$ be a factor of $\phi$ with factor map $p$. There holds
\[\operatorname{E}(\phi \overset{p}{\to} \psi)=\sup_{\eta \in \mathbb{B}_X} \limsup_{i\in I} \frac{\log(\operatorname{spa}_p(\eta_{A_i}))}{\mu(A_i)}=\sup_{\eta \in \mathbb{B}_X} \limsup_{i\in I} \frac{\log(\operatorname{sep}_p(\eta_{A_i}))}{\mu(A_i)}\]
% \begin{align*}
%\operatorname{E}(\phi \overset{p}{\to} \psi)=\sup_{\eta \in \mathbb{B}_X} \limsup_{i\in I} \frac{\log(\operatorname{spa}_p(\eta_{A_i}))}{\mu(A_i)}
%
%
%
%&=\sup_{\eta \in \mathbb{B}_X} \liminf_{i\in I} \frac{\log(\operatorname{spa}_p(\eta_{A_i}))}{\mu(A_i)}
%=\sup_{\eta \in \mathbb{B}_X} \limsup_{i\in I} \frac{\log(\operatorname{spa}_p(\eta_{A_i}))}{\mu(A_i)}\\
%&=\sup_{\eta \in \mathbb{B}_X} \liminf_{i\in I} \frac{\log(\operatorname{sep}_p(\eta_{A_i}))}{\mu(A_i)}
%=\sup_{\eta \in \mathbb{B}_X} \limsup_{i\in I} \frac{\log(\operatorname{sep}_p(\eta_{A_i}))}{\mu(A_i)},
%\end{align*}		
		for any  Van Hove net $(A_i)_{i\in I}$ in $G$ and any base $\mathbb{B}_X$ of $\mathbb{U}_X$. Furthermore the statement remains valid after replacing the limit superior by a limit inferior.
\end{theorem}

\begin{proof}
	As $\limsup_{i\in I} {\log(\operatorname{spa}_p(\eta_{A_i}))}/{\mu(A_i)}$ and the other similar terms are increasing, whenever $\eta$ is decreasing with respect to set inclusion, it suffices to show the statement for $\mathbb{B}_X=\mathbb{U}_X$. By Lemma \ref{lem:SpaSepCovNONStaticRelation} it is immediate that for any $\eta\in \mathbb{U}_X$ there holds 
	\begin{align*}
		\operatorname{E}(\eta|\phi\overset{p}{\to}\psi)
			&\leq \sup_{\theta\in \mathbb{U}_X} \limsup_{i\in I} \frac{\log(\operatorname{spa}_p(\theta_{A_i}))}{\mu(A_i)}
			%\\
			%&\leq \sup_{\theta\in \mathbb{U}_X} \limsup_{i\in I} \frac{\log(\operatorname{sep}_p(\theta_{A_i}))}{\mu(A_i)}
			\leq \operatorname{E}(\phi\overset{p}{\to}\psi).
 	\end{align*}
	Taking the supremum over $\eta$ yields the first equality. Similar arguments show the statements about $\operatorname{sep}$ and the limit inferior. 
\end{proof}

\section{Relative topological entropy via lattices}\label{sec:Latticerestriction}

	In this section we provide a proof for Theorem \ref{theint:linkoflatticetogroup}. 
	Recall that for a map $f\colon A \to B$ and $M\subseteq A$ we denote by $f\big|_M$ the restriction $f\big|_M\colon M\to B\colon a\mapsto f(a)$. As before we assume $G$ to satisfy the Ornstein-Weiss lemma. 

\begin{proposition}\label{pro:MittelnalongDelone}
	Let $\Lambda$ be a relatively dense subset of $G$ and let $(A_i)_{i\in I}$ be a Van Hove net. Set $F_i:=A_i\cap \Lambda$. Let furthermore $f\colon \mathcal{K}(G)\to \mathbb{R}$ be a subadditive, right invariant and monotone mapping. If $\Lambda$ has a well defined uniform density $\operatorname{dens}(\Lambda)$, then there holds 
	\[\lim_{i\in I}\frac{f(A_i)}{\mu(A_i)}=\operatorname{dens}(\Lambda)\lim_{i\in I}\frac{f(K F_i)}{|F_i|}.\]
\end{proposition}
\begin{proof}
Let $K$ be a compact and symmetric subset of $G$ that contains $e_G$ and such that $K\Lambda=G$. Let furthermore $M$ be a compact and symmetric neighbourhood of $K$. Then there holds
\[1\geq \frac{\mu\left(\overline{A_i\setminus \partial_M A_i}\right)}{\mu(A_i)}\geq \frac{\mu(A_i\setminus \partial_M A_i)}{\mu(A_i)}\geq \frac{\mu(A_i)-\mu(\partial_M A_i)}{\mu(A_i)}\rightarrow 1. \]
	Thus $\lim_{i\in I}{\mu\left(\overline{A_i\setminus \partial_M A_i}\right)}/{\mu(A_i)}=1$. Furthermore a straight forward argument shows $(\overline{A_i\setminus \partial_M A_i})^c\subseteq ({A_i\setminus \partial_M A_i})^c= M\overline{A_i^c}$. Thus for $C\subseteq G$ compact we obtain 
	\[\partial_C\overline{A_i\setminus \partial_M A_i}=C\overline{A_i\setminus \partial_M A_i}\cap C\overline{\overline{A_i\setminus \partial_M A_i}^c}\subseteq C A_i \cap CM\overline{A_i^c}\subseteq \partial_{CM} A_i\]
	and we obtain $(\overline{A_i\setminus \partial_M A_i})_{i\in I}$ to be a Van Hove net in $G$. 
	As $M$ is a compact neighbourhood we obtain $\overline{A_i\setminus \partial_M A_i}\subseteq A_i\setminus \partial_K A_i$ and one easily shows $A_i\setminus \partial_K A_i\subseteq KF_i$. From Proposition \ref{pro:VanHovenets+K} we know that $(KA_i)_{i\in I}$ is a Van Hove net in $G$ with $\lim_{i\in I}\mu(KA_i)/\mu(A_i)=1$ and compute  
\begin{align*}
	\lim_{i\in I}\frac{f\left(A_i\right)}{\mu\left(A_i\right)}&=\lim_{i\in I}\frac{f\left(\overline{A_i\setminus \partial_M A_i}\right)}{\mu\left(\overline{A_i\setminus \partial_M A_i}\right)}
	\leq\liminf_{i\in I}\frac{f(KF_i)}{\mu(A_i)}\\
	&\leq\limsup_{i\in I}\frac{f(KF_i)}{\mu(A_i)}
	\leq \lim_{i\in I}\frac{f(KA_i)}{\mu(KA_i)}
	=\lim_{i\in I}\frac{f(A_i)}{\mu(A_i)}.
\end{align*}
Thus the statement follows from $\operatorname{dens}(\Lambda)=\lim_{i\in I}\frac{|F_i|}{\mu(A_i)}$.
%		
% A straight forward argument shows $F_i\subseteq A_i\subseteq \partial_K A_i \cup KF_i$ for all $i\in I$. From Proposition \ref{pro:VanHovenets+K} we know that $(KA_i)_{i\in I}$ is a Van Hove net in $G$ with $\lim_{i\in I}\mu(KA_i)\mu(A_i)^{-1}=1$ and compute  
%\begin{align*}
%	\lim_{i\in I}\frac{f(A_i)}{\mu(A_i)}&\leq
%	\lim_{i\in I}\frac{f(\partial_K A_i)}{\mu(A_i)}+\liminf_{i\in I}\frac{f(KF_i)}{\mu(A_i)}\\
%	&\leq 	0+\limsup_{i\in I}\frac{f(KF_i)}{\mu(A_i)}\\
%	&\leq \limsup_{i\in I}\frac{f(KA_i)}{\mu(A_i)}
%	=\lim_{i\in I}\frac{f(KA_i)}{\mu(KA_i)}
%	=\lim_{i\in I}\frac{f(A_i)}{\mu(A_i)}.
%\end{align*}
%Thus the statement follows from $\operatorname{dens}(\Lambda)=\lim_{i\in I}\frac{|F_i|}{\mu(A_i)}$.
\end{proof}

 \begin{theorem}\label{the:linkoflatticetogroup}
Let $\phi$ be an action of $G$ on a compact Hausdorff space $X$. Let furthermore $\psi$ be a factor of $\phi$ via factor map $p\colon X\to Y$. Let $\Lambda$ be a relatively dense subset of $G$ and let $(A_i)_{i\in I}$ be a Van Hove net. Set $F_i:=A_i\cap \Lambda$.
	\begin{itemize}
		\item[(i)] If $\Lambda$ has a well defined uniform density $\operatorname{dens}(\Lambda)$, then there holds 
\[\operatorname{E}(\phi\overset{p}{\to}\psi)=\operatorname{dens}(\Lambda)\sup_{\eta\in \mathbb{U}_X}\liminf_{i\in I}\frac{\log(\operatorname{cov}_p(\eta_{F_i}))}{|F_i|}=\operatorname{dens}(\Lambda)\sup_{\eta\in \mathbb{U}_X}\limsup_{i\in I}\frac{\log(\operatorname{cov}_p(\eta_{F_i}))}{|F_i|}\]
and these statements remain valid if we consider $\operatorname{sep}_p$ and $\operatorname{spa}_p$.
	\item[(ii)] If $\Lambda$ is a uniform lattice, then there holds
	\[\operatorname{E}(\phi\overset{p}{\to}\psi)=\operatorname{dens}(\Lambda)\operatorname{E}\left(\phi\big|_{\Lambda\times X}\overset{p}{\to} \psi\big|_{\Lambda\times Y}\right).\]
	\end{itemize}
\end{theorem} 

\begin{proof}
The second equality in (i) follows from the first as in the proof of Theorem \ref{the:Entropyingenerality}. 
To show the first equality consider a compact subset $K\subseteq G$ with $K\Lambda=G$. From Proposition \ref{pro:MittelnalongDelone} we obtain	
	\begin{align*}
	\operatorname{E}\left(\phi\overset{p}{\to} \psi\right)&=\sup_{\eta\in \mathbb{U}_X}\lim_{i\in I} \frac{\log(\operatorname{cov}_p(\eta_{A_i}))}{\mu(A_i)}\\
	&=	\sup_{\eta\in \mathbb{U}_X}\lim_{i\in I} \frac{\log(\operatorname{cov}_p(\eta_{K F_i}))}{|F_i|}\operatorname{dens}(\Lambda) \\
		&\leq \sup_{\eta\in \mathbb{U}_X} \lim_{i\in I} \frac{\log(\operatorname{cov}_p((\eta_K)_{F_i}))}{|F_i|}\operatorname{dens}(\Lambda) \\
	&\leq \sup_{\kappa\in \mathbb{U}_X}\liminf_{i\in I} \frac{\log(\operatorname{cov}_p(\kappa_{F_i}))}{|F_i|}\operatorname{dens}(\Lambda) \\
		&\leq \sup_{\kappa\in \mathbb{U}_X}\limsup_{i\in I} \frac{\log(\operatorname{cov}_p(\kappa_{F_i}))}{|F_i|}\operatorname{dens}(\Lambda) \\
	%&\leq \sup_{\kappa\in \mathbb{U}_X}\liminf_{i\in I} \frac{\log(\operatorname{cov}_p(\kappa_{A_i}))}{|F_i|}\operatorname{dens}(\Lambda) \\
	&\leq \sup_{\kappa\in \mathbb{U}_X}\lim_{i\in I} \frac{\log(\operatorname{cov}_p(\kappa_{A_i}))}{\mu(A_i)}. 
	\end{align*}
	
%	To show (ii) it suffices to prove that $(F_i)_{i\in I}$ is a Van Hove net. Let $C\subseteq G$ be a Fundamental domain of $\Lambda$ and set $K:=\overline{C}$.
%	 Denote by $\partial^\Lambda_E F_i$ the Van Hove boundary with respect to $\Lambda$ for some finite set $E\subseteq \Lambda$. Then a similar argument as in the proof of Lemma \ref{lem:constructingFhatandFcheck} yields $\partial^\Lambda_E F_i\subseteq \partial_E^G C F_i\subseteq \partial_E^G C A_i.$ 
%	Hence \[0\leq \operatorname{dens}(\Lambda)\limsup_{i\in I}\frac{|\partial_E^\Lambda F_i|}{|F_i|}\leq \limsup_{i\in I}\frac{\mu(C\partial_E^G C A_i)}{\mu(A_i)}=%\leq \limsup_{i\in I} \frac{\mu(\partial_{KE}^G K A_i)}{\mu(A_i)}=
%	\lim_{i\in I}\frac{\mu(\partial_{KE}^G K A_i)}{\mu(KA_i)}=0.\]
%
%% 
%------------
This shows (i). To show (ii) note first that every uniform lattice has a well defined uniform density which is given by $\operatorname{dens}(\Lambda)=\mu(C)^{-1}$ for any fundamental domain $C\subseteq G$ of $\Lambda$. We use 
Theorem \ref{the:OrnsteinWeisslemma} to obtain that for any Van Hove net $(E_i)_{i\in I}$ in $\Lambda$ there holds
$\mu(C)\operatorname{E}(\phi\overset{p}{\to}\psi)= \sup_{\eta \in \mathbb{U}_X} \lim_{i\in I}  {\log(\operatorname{cov}_p(\eta_{\overline{C}E_i}))}{|E_i|}^{-1}$.
As
\begin{align*}
\operatorname{E}\left(\phi\big|_{\Lambda\times X}\overset{p}{\to} \psi\big|_{\Lambda\times Y}\right)&= \sup_{\eta \in \mathbb{U}_X} \lim_{i\in I}  \frac{\log(\operatorname{cov}_p(\eta_{\overline{C}E_i}))}{|E_i|}
\end{align*}
follows similarly as in (i) we obtain the statement. 				
\end{proof}

\begin{remark}
Note that for $n\in \mathbb{N}$ the set $\{0,\cdots,n-1\}$ is a fundamental domain for the uniform lattice $n \mathbb{Z}$ in $\mathbb{Z}$. We thus obtain from Theorem \ref{the:linkoflatticetogroup} for every homeomorphism $f\colon X\to X$ the well known formula $n \operatorname{E}(f)= \operatorname{E}(f^n). $
	Here  $\operatorname{E}(g)$ abbreviates the topological entropy of the flow $\phi\colon \mathbb{Z}\times X\to X$ with $\phi(n,x)=g^n(x)$ for a homeomorphism $g\colon X\to X$. Thus $\operatorname{E}(f^n)$ is the entropy of the flow $(m,x)\mapsto f^m(x)$ restricted to $n \mathbb{Z}\times X\to X$. 
\end{remark}

\section{Bowen's formula}\label{sec:The Bowen entropy formula for actions of groups that contain a uniform lattice}

In this section we present a proof of Theorem \ref{theint:Bowenstheoremgeneral}. The following consequences can be drawn directly from this theorem and the literature \cite{Yan}.
%We will first see that the result follows from the variational principle and the Rohlin-Abramov theorem, which are at hand in the context of countable discrete groups and Theorem \ref{the:linkoflatticetogroup}.  
	
	\begin{remark}
%		Taking $\rho$ as the action on a one point space in the first inequality yields, that
%		$\max\{\operatorname{E}(\phi \overset{p}{\to}\psi),\operatorname{E}(\psi)\} \leq \operatorname{E}(\phi) .$

 If we take $\rho$ as the action on a one point space, we obtain the classical formulation of Bowen's formula for the entropy of factors from the second inequality.
Furthermore we obtain $\operatorname{E}(\phi)=\operatorname{E}(\phi \overset{p}{\to}\psi)$, whenever $\operatorname{E}(\psi)=0$; and $\operatorname{E}(\phi)=\operatorname{E}(\psi)$, whenever $\operatorname{E}(\phi \overset{p}{\to}\psi)=0$. 
If we assume $G$ to be non-compact and  $X$ and $Y$ to be compact metrizable spaces the latter assumption $\operatorname{E}(\phi \overset{p}{\to} \psi)=0$ is satisfied under each of the following conditions. 
\begin{itemize}
	\item[(i)] $p$ is a distal factor map, i.e. for $y\in Y$ all pairs of distinct points in $p^{-1}(y)$ are distal\footnote{
We call two points $x,x'\in X$ \emph{distal}, whenever there is $\eta\in \mathbb{U}_X$ such that  $(x,x')\notin \eta_g$ for all $g\in G$. 
	}. 
	\item[(ii)] $p$ is a countable to one factor map, i.e. for $y\in Y$ the fiber $p^{-1}(y)$ is countable.
\end{itemize}
\begin{proof}
	The statement in (i) implies $p$ to be a distal factor map from the action $\phi\big|_{\Lambda\times X}$ to $\psi\big|_{\Lambda\times Y}$ for every countable uniform lattice $\Lambda\subseteq G$. We thus obtain the statement from \cite[Corollary 6.7]{Yan} and Theorem \ref{the:linkoflatticetogroup}. We get (ii) directly from \cite[Theorem 5.7]{Yan} and Theorem \ref{the:linkoflatticetogroup}. Note that we restrict to metric spaces, as the statements in \cite{Yan} are presented for compact metric spaces. 
\end{proof}

	\end{remark}
	
	%We will obtain these results, by introducing measure theoretical entropy for actions of discrete amenable groups. This approach follows \cite{Yan}. Note that there is no straight forward generalization of the related definitions to non-discrete groups. In \cite[Theorem 5.1]{Yan} it is shown, that for actions of countable amenable discrete groups on compact metric spaces the variational principle is valid. We will see, that this statement can be easily generalized to the setting of discrete amenable groups acting on compact Hausdorff spaces. Furthermore in \cite[Theorem 3.1]{Yan} it is shown, that the Rohlin-Abramov theorem (see Proposition \ref{pro:RohlinAbramov} below) is valid. The combination of these results with Theorem \ref{the:linkoflatticetogroup} will then yield Theorem \ref{the:Bowenstheoremgeneral}. Note furthermore, that it is an open question, whether one can define measure theoretical relative entropy for non-discrete amenable group actions and establish a variational principle.

\subsection{Measure theoretic Relative Entropy for actions of countable discrete amenable groups}
	We give a brief introduction into the theory of measure theoretical relative entropy of actions of discrete amenable groups, presented in various texts, such as \cite{ward1992abramov,danilenko2001entropy,weiss2003actions,ollagnier2007ergodic,Yan,downarowicz2019symbolic}, in order to state the variational principle and the Rohlin-Abramov theorem. 
	Let $X$ be a compact Hausdorff space. By $\mathcal{B}_X$ we denote the Borel $\sigma$-algebra. Furthermore we denote the set of all regular Borel probability measures by $\mathcal{M}(X)$. By the Riesz-Markov theorem we can identify $\mathcal{M}(X)$ with the convex cone base of all positive functionals in $C(X)$ that map the unit $(X\to \mathbb{R};x \mapsto 1)$ to $1$. We equip $\mathcal{M}(X)$ with the restricted weak*-topology and obtain a compact topological space from the Banach-Alaoglu theorem. For a reference see \cite[Theorem E.11]{eisner2015operator}.
%	
%	As this set is closed and contained in the weak*-compact unit ball, we can equip $\mathcal{M}(X)$ with the restricted weak*-topology, to obtain a compact topological space.
%	
%	
%	set of all (positive) Borel measures with the cone of all positive functionals on the Banach space $C(X)$.  
%For a reference see \cite[Theorem E.11]{eisner2015operator}.
% The set $\mathcal{M}(X)$ can be identified with the convex cone base of all positive functionals in $C(X)$ that map the unit $(X\to \mathbb{R};x \mapsto 1)$ to $1$. As this set is closed and contained in the weak*-compact unit ball, we can equip $\mathcal{M}(X)$ with the restricted weak*-topology, to obtain a compact topological space. 
	
	 A family $\alpha$ of pairwise disjoint Borel-measurable non empty subsets of $X$ with union $X$ is called a \emph{measurable partition} of $X$. We denote the set of all finite partitions of $X$ by $\mathcal{P}_X$. The \emph{refinement} of two partitions $\alpha,\beta\in \mathcal{P}_X$ is the partition $\alpha \vee \beta:=\{A\cap B;\, A\in \alpha\text{ and }B \in \beta\}\setminus \{\emptyset\}$. Similarly the refinement of a finite number of partitions is defined. %$\alpha_1,\cdots,\alpha_n$ and denote this partition by $\bigvee_{i=1}^n\alpha_i$. 
	Let  $\phi\colon \Lambda\times X \to X$ be an action of a countable discrete amenable group $\Lambda$ on a compact Hausdorff space. For a finite subset $F\subseteq \Lambda$ we denote by $\alpha_F$ the refinement of the partitions $\{(\phi^g)^{-1}(A);\, A\in \alpha\}$, where $g$ ranges over $F$. 	
	A measure $\nu \in \mathcal{M}(X)$ is called $\phi$- invariant, if $\nu(A)=\nu(\phi^g(A))$ for every $g\in G$. We denote by $\mathcal{M}_\phi$ the set of all $\phi$-invariant $\nu \in \mathcal{M}(X)$. %If $p\colon X \to Y$ is a factor map to some factor $\psi\colon \Lambda\times Y\to Y$ that acts on a compact Hausdorff space $Y$ and $(\nu,\nu')\in \mathcal{M}_\phi(X)\times \mathcal{M}_\psi(Y)$ we say that $p$ is 
		 Every continuous map $p\colon X \to Y$ to some compact Hausdorff space is measurable with respect to the Borel $\sigma$-algebras and $p^{-1}(\mathcal{B}_Y)$ is a sub-$\sigma$-algebra of $\mathcal{B}_X$. For $A\in \mathcal{B}_X$ and $\nu\in \mathcal{M}_\phi$ let $\mathbb{E}_{\nu,p}(\chi_A)$ be the conditional expectation of the characteristic function $\chi_A$ of $A$ with respect to $p^{-1}(\mathcal{B}_X)$.	
	For $\alpha \in \mathcal{P}_X$	we define 
\[H_{\nu,p}(\alpha):=- \sum_{A\in \alpha } \int_X  \mathbb{E}_{\nu,p}(\chi_A)\log(\mathbb{E}_{\nu,p}(\chi_A))d\nu.\]
	As presented in \cite{Yan} 
	%$p^{-1}(\mathcal{B}_X)$ is a invariant under $\phi$ up to sets of measure $0$, and every amenable discrete group is an Ornstein-Weiss group, we consider 
	the Ornstein-Weiss lemma can be applied to $\mathcal{F}(\Lambda)\ni  F\mapsto H_{\nu,p}(\alpha_F)$ for every $\alpha\in \mathcal{P}_X$ to obtain that
 $\operatorname{E}_\nu(\alpha|\phi\overset{p}{\to}\psi):=\lim_{i\in I} {H_{\nu,p}(\alpha_{F_i})}{|F_i|}^{-1}$
	exists and is independent of the choice of the Van Hove net $(F_i)_{i\in I}$ in $\Lambda$. The \emph{relative measure theoretical entropy} of $\phi$ under the condition $\psi$ is given by \[\operatorname{E}_\nu(\phi\overset{p}{\to}\psi):=\sup_{\alpha \in \mathcal{P}_X} \operatorname{E}_\nu(\alpha|\phi\overset{p}{\to}\psi).\]
	The Rohlin-Abramov theorem for actions of countable amenable groups appeared firstly in \cite{ward1992abramov} and can be also found in \cite{glasner2000entropy,danilenko2001entropy,dooley2015local,Yan}. For a continuous surjective mapping $p\colon X \to Y$ between compact Hausdorff spaces and a $\nu \in \mathcal{M}(X)$ the \emph{push forward measure} is defined by $p_*\nu\in \mathcal{M}(Y)$ by $p_*\nu(M):=\nu(p^{-1}(M))$ for $M\in \mathcal{B}_Y$. %This restricts to a mapping $p_*\colon \mathcal{M}(X)\to \mathcal{M}(Y)$. 
	
	\begin{proposition} (Rohlin-Abramov Theorem)%\footnote{TO DO: This is there only stated for countable group actions, hence yields the result only for amenable groups with countable uniform lattices. Check for not countable groups!!!!}
	\label{pro:RohlinAbramov}
		Let $\phi$, $\psi$ and $\rho$ be actions of a countable discrete amenable group $\Lambda$ on compact Hausdorff spaces $X$, $Y$ and $Z$ respectively. Let $\psi$ be a factor of $\phi$ via the factor map $p$ and $\rho$ be a factor of $\psi$ via the factor map $q$, i.e.\ $\phi \overset{p}{\to}\psi\overset{q}{\to}\rho$. Then for $(\nu_X,\nu_Y,\nu_Z)\in \mathcal{M}_\phi(X)\times \mathcal{M}_\psi(Y)\times \mathcal{M}_\rho(Z)$ that satisfy $p_*\nu_X=\nu_Y$ and $q_*\nu_Y=\nu_Z$ there holds
		\begin{align*}
		\operatorname{E}_{\nu_X}(\phi \overset{q\circ p}{\to}\rho)= \operatorname{E}_{\nu_X}(\phi \overset{p}{\to}\psi)+\operatorname{E}_{\nu_Y}(\psi \overset{q}{\to}\rho). 
		\end{align*}
	\end{proposition}

The following version of a variational principle is valid in our context. 
\begin{theorem} \label{the:Variationalprinciple}
	Let $\phi$ be an action of an amenable group $G$, containing a countable uniform lattice $\Lambda$, on a compact Hausdorff space $X$ and let $\psi$ be a factor of $\phi$ via $p\colon X \to Y$. Then 
	\[\operatorname{E}(\phi \overset{p}{\to} \psi)=\mu(C)\sup_{\nu \in \mathcal{M}_{\phi|_{\Lambda\times X}}}\operatorname{E}_\nu\left(\phi\big|_{\Lambda\times X}\overset{p}{\to} \psi\big|_{\Lambda\times Y}\right).\]
%	where $C\subseteq G$ is a fundamental domain of $\Lambda$. %, $\mu$ is a Haar measure on $G$ 
%	and $\phi^\Lambda$ and $\psi^\Lambda$ are the restricted action of $\phi$ to $\Lambda \times X$ and of $\psi$ to $\Lambda\times Y$ respectively. 
\end{theorem}

\begin{proof}
By  Theorem \ref{the:linkoflatticetogroup} it remains to show that 
\[\operatorname{E}\left(\phi\big|_{\Lambda\times X}\overset{p}{\to} \psi\big|_{\Lambda\times Y}\right)=\sup_{\nu \in \mathcal{M}_{\phi|_{\Lambda\times X}}}\operatorname{E}_\nu\left(\phi\big|_{\Lambda\times X}\overset{p}{\to} \psi\big|_{\Lambda\times Y}\right).\]
A proof can be found in \cite{ollagnier1982variational,stepin1980variational} and goes back to \cite{goodwyn1969topological,dinaburg1971connection,goodman1971relating}. 
%
%As presented in section \ref{sec:Actions on metric and topological spaces} our definition of relative topological entropy is equivalent to the definition given in \cite{Yan}.
 %Note that  \cite[Lemma 5.4]{Yan} is also valid for compact Hausdorff spaces. Thus the proof given in \cite[Theorem 5.1]{Yan}  generalizes to actions on such spaces. %actions on compact Hausdorff spaces.% and we obtain the needed statement from there, as $\Lambda$ is a discrete and countable group.
		% Note furthermore, that our definition of Van Hove nets coincides with the definition of F\o lner nets in the context of discrete groups. 
%		Furthermore a slight generalization of \cite[Remark 2.5]{Yan} to nets shows, that our definition of relative topological entropy coincides with the definition given in \cite{Yan}. In \cite[Theorem 5.1]{Yan} assumption of countability is not necessary. This can be seen by replacing sequences by nets in the proof. 		 
		%The generalization of the rest of the proof is straightforward. %Note that the equivalence of the definition given in \cite{Yan} and our definition is shown in \ref{rem:Yansentropyisequivalent}.	
\end{proof}

For the proof of Theorem \ref{theint:Bowenstheoremgeneral} we need a further ingredient. 
By a standard Krylov–Bogolyubov procedure one obtains the following. 
	
	\begin{lemma} \label{lem:KrylovBogoliubovgeneralized}
		Let $\phi$ be an action of a discrete amenable group $\Lambda$ on a compact topological space $X$ and $\psi$ be a factor of $\phi$ via factor map $p$. Then the restricted push forward operation $p_*\colon \mathcal{M}_\phi\to  \mathcal{M}_\psi$ is surjective.%, i.e.\ for any $\tilde{\nu} \in \mathcal{M}_\psi$ there is $\nu\in \mathcal{M}_\phi$, such that $\tilde{\nu}=p_*\nu$. 
	\end{lemma}

\subsection{Proof of Theorem \ref{theint:Bowenstheoremgeneral}}

\begin{proof}[Proof of Theorem \ref{theint:Bowenstheoremgeneral}] 

Let $\phi$, $\psi$ and $\rho$ be actions of an amenable group containing a countable uniform lattice $\Lambda$ on compact Hausdorff spaces $X$, $Y$ and $Z$ respectively. Assume $\phi \overset{p}{\to}\psi\overset{q}{\to}\rho$. Let $C$ be a fundamental domain of $\Lambda$. We abbreviate $\phi^\Lambda:=\phi\big|_{\Lambda\times X}$, $\psi^\Lambda$ %:=\psi\big|_{\Lambda\times Y}$ 
and $\rho^\Lambda$ %:=\rho\big|_{\Lambda\times Z}$ 
for the restrictions of the actions to the lattice. From Theorem \ref{the:Variationalprinciple} and Proposition \ref{pro:RohlinAbramov} we obtain
\begin{align*}
		\operatorname{E}(\phi \overset{p}{\to}\psi)&=\mu(C)\sup_{\nu\in \mathcal{M}_{\phi^\Lambda}}\operatorname{E}_\nu(\phi^\Lambda\overset{p}{\to}\psi^\Lambda)\\
		&\leq \mu(C)\sup_{\nu\in \mathcal{M}_{\phi^\Lambda}}\left(\operatorname{E}_\nu(\phi^{\Lambda}\overset{p}{\to}\psi^\Lambda)+\operatorname{E}_{p_*\nu}(\psi^\Lambda\overset{q}{\to}\rho^\Lambda)\right)\\	
		&= \mu(C)\sup_{\nu\in \mathcal{M}_{\phi^\Lambda}}\operatorname{E}_\nu(\phi^\Lambda\overset{q\circ p}{\to}\psi^\Lambda)
		= \operatorname{E}(\phi \overset{q\circ p}{\to}\rho). 
		\end{align*}
By Lemma \ref{lem:KrylovBogoliubovgeneralized} we observe 
\[\operatorname{E}(\psi^\Lambda\overset{q}{\to}\rho^\Lambda)=\mu(C)\sup_{\omega\in \mathcal{M}_{\psi^\Lambda}}\operatorname{E}_\omega(\psi^\Lambda\overset{q}{\to}\rho^\Lambda)= \mu(C)\sup_{\nu\in \mathcal{M}_{\phi^\Lambda}}\left(\operatorname{E}_{p_*\nu}(\psi^\Lambda\overset{q}{\to}\rho^\Lambda)\right)\]
and argue similarly to obtain $\operatorname{E}(\psi \overset{q}{\to}\rho) \leq \operatorname{E}(\phi \overset{q\circ p}{\to}\rho)$.
 To finish the proof we compute
\begin{align*}
		\operatorname{E}(\phi \overset{q\circ p}{\to}\rho)&= \mu(C)\sup_{\nu\in \mathcal{M}_{\phi^\Lambda}}\operatorname{E}_\nu(\phi^\Lambda\overset{q\circ p}{\to}\psi^\Lambda)\\
		&=\mu(C)\sup_{\nu\in \mathcal{M}_{\phi^\Lambda}}\left(\operatorname{E}_\nu(\phi^\Lambda\overset{p}{\to}\psi^\Lambda)+\operatorname{E}_{p_*\nu}(\psi^\Lambda\overset{q}{\to}\rho^\Lambda)\right)\\	
		&\leq \mu(C)\sup_{\nu\in \mathcal{M}_{\phi^\Lambda}}\left(\operatorname{E}_\nu(\phi^\Lambda\overset{p}{\to}\psi^\Lambda)\right)+\mu(C)\sup_{\omega\in \mathcal{M}_{\psi^\Lambda}}\left(\operatorname{E}_\omega(\psi^\Lambda\overset{q}{\to}\rho^\Lambda)\right)\\	
		&=\operatorname{E}(\phi \overset{p}{\to}\psi)+\operatorname{E}(\psi \overset{q}{\to}\rho).
		\end{align*}
\end{proof}

\begin{acknowledgement}
The author would like to thank Yves de Cornulier for pointing out Example \ref{exa:CPS} to him and is grateful for lively discussions with Felix H. Pogorzelski on the topic of approximate lattices. Furthermore the author would like to thank Tobias J\"ager for his patient supervision and several useful hints.  
\end{acknowledgement}

%\begin{acknowledgement}
%\end{acknowledgement}

\footnotesize
\bibliographystyle{spbasic}

\begin{thebibliography}{EFHN15}

\bibitem[AG95]{axel1995beyond}
Francoise Axel and D~Gratias.
\newblock Beyond quasicrystals.
\newblock 1995.

\bibitem[Baa00]{baake2000directions}
Michael Baake.
\newblock {\em Directions in mathematical quasicrystals}.
\newblock Number~13. American Mathematical Soc., 2000.

\bibitem[BG13]{baake2013aperiodic}
Michael Baake and Uwe Grimm.
\newblock {\em Aperiodic order}, volume~1.
\newblock Cambridge University Press, 2013.

\bibitem[BH15]{baake2015ergodic}
Michael Baake and Christian Huck.
\newblock Ergodic properties of visible lattice points.
\newblock {\em Proceedings of the Steklov Institute of Mathematics},
  288(1):165--188, 2015.

\bibitem[BH{\etalchar{+}}18]{bjorklund2018approximate}
Michael Bj{\"o}rklund, Tobias Hartnick, et~al.
\newblock Approximate lattices.
\newblock {\em Duke Mathematical Journal}, 167(15):2903--2964, 2018.

\bibitem[BHP18]{bjorklund2018aperiodic}
Michael Bj{\"o}rklund, Tobias Hartnick, and Felix Pogorzelski.
\newblock Aperiodic order and spherical diffraction, i: auto-correlation of
  regular model sets.
\newblock {\em Proceedings of the London Mathematical Society},
  116(4):957--996, 2018.

\bibitem[BL04]{baake2004dynamical}
Michael Baake and Daniel Lenz.
\newblock Dynamical systems on translation bounded measures: Pure point
  dynamical and diffraction spectra.
\newblock {\em Ergodic Theory and Dynamical Systems}, 24(6):1867--1893, 2004.

\bibitem[BL17]{baake2016spectral}
Michael Baake and Daniel Lenz.
\newblock Spectral notions of aperiodic order.
\newblock {\em Discrete \& Continuous Dynamical Systems-S}, 10(2):161--190,
  2017.

\bibitem[BLR07]{baake2007pure}
Michael Baake, Daniel Lenz, and Christoph Richard.
\newblock Pure point diffraction implies zero entropy for delone sets with
  uniform cluster frequencies.
\newblock {\em Letters in Mathematical Physics}, 82(1):61--77, 2007.

\bibitem[Bow71]{bowen1971entropy}
Rufus Bowen.
\newblock Entropy for group endomorphisms and homogeneous spaces.
\newblock {\em Transactions of the American Mathematical Society},
  153:401--414, 1971.

\bibitem[BS02]{BrinandStuck}
Michael Brin and Garrett Stuck.
\newblock {\em Introduction to dynamical systems}.
\newblock Cambridge university press, 2002.

\bibitem[CdLH14]{cornulier2014metric}
Yves Cornulier and Pierre de~La~Harpe.
\newblock Metric geometry of locally compact groups.
\newblock Winner of the 2016 EMS Monograph Award. EMS Tracts in Mathematics, 25. European Mathematical Society (EMS), Zürich, 2016. viii+235 pp. ISBN: 978-3-03719-166-8.

\bibitem[CSCK14]{FeketesLemma}
Tullio Ceccherini-Silberstein, Michel Coornaert, and Fabrice Krieger.
\newblock An analogue of fekete’s lemma for subadditive functions on
  cancellative amenable semigroups.
\newblock {\em Journal d'analyse math{\'e}matique}, 124(1):59--81, 2014.

\bibitem[Dan01]{danilenko2001entropy}
Alexandre~I Danilenko.
\newblock Entropy theory from the orbital point of view.
\newblock {\em Monatshefte f{\"u}r Mathematik}, 134(2):121--141, 2001.

\bibitem[DE14]{deitmar2014principles}
Anton Deitmar and Siegfried Echterhoff.
\newblock {\em Principles of harmonic analysis}.
\newblock Springer, 2014.




\bibitem[DHZ19]{downarowicz2019tilings}
Tomasz Downarowicz, Dawid Huczek, and Guohua Zhang.
\newblock Tilings of amenable groups.
\newblock {\em Journal f{\"u}r die reine und angewandte Mathematik (Crelles
  Journal)}, 2019(747):277--298, 2019.

\bibitem[Din71]{dinaburg1971connection}
Efim~I Dinaburg.
\newblock A connection between various entropy characterizations of dynamical
  systems.
\newblock {\em Izv. Akad. Nauk SSSR Ser. Mat}, 35(324-366):13, 1971.

\bibitem[DS58]{DunfordandSchwartz}
Nelson Dunford and Jacob~T Schwartz.
\newblock {\em Linear operators part I: general theory}, volume~7.
\newblock Interscience publishers New York, 1958.

\bibitem[DSV12]{DikranjanandSanchis}
Dikran Dikranjan, Manuel Sanchis, and Simone Virili.
\newblock New and old facts about entropy in uniform spaces and topological
  groups.
\newblock {\em Topology Appl}, 159(7):1916--1942, 2012.



\bibitem[DZ15]{dooley2015local}
Anthony Dooley and Guohua Zhang.
\newblock {\em Local entropy theory of a random dynamical system}, volume 233.
\newblock American Mathematical Society, 2015.

\bibitem[DZ19]{downarowicz2019symbolic}
Tomasz Downarowicz and Guohua Zhang.
\newblock Symbolic extensions of amenable group actions and the comparison
  property.
\newblock {\em arXiv preprint arXiv:1901.01457}, 2019.
\newblock Accepted for publication in Memoirs of the American Mathematical Society.



\bibitem[EFHN15]{eisner2015operator}
Tanja Eisner, B{\'a}lint Farkas, Markus Haase, and Rainer Nagel.
\newblock {\em Operator theoretic aspects of ergodic theory}, volume 272.
\newblock Springer, 2015.

\bibitem[FGJO18]{fuhrmann2018irregular}
Gabriel Fuhrmann, Eli Glasner, Tobias J{\"a}ger, and Christian Oertel.
\newblock Irregular model sets and tame dynamics.
\newblock {\em arXiv preprint arXiv:1811.06283}, 2018.

\bibitem[Fol13]{Folland}
Gerald~B Folland.
\newblock {\em Real analysis: modern techniques and their applications}.
\newblock John Wiley \& Sons, 2013.


\bibitem[GTW00]{glasner2000entropy}
Eli Glasner, J-P Thouvenot, and Benjamin Weiss.
\newblock Entropy theory without a past.
\newblock {\em Ergodic Theory and Dynamical Systems}, 20(5):1355--1370, 2000.


  
  \bibitem[Goo69]{goodwyn1969topological}
L~Wayne Goodwyn.
\newblock Topological entropy bounds measure-theoretic entropy.
\newblock {\em Proceedings of the American Mathematical Society},
  23(3):679--688, 1969.
  
  \bibitem[Goo71]{goodman1971relating}
Tim~NT Goodman.
\newblock Relating topological entropy and measure entropy.
\newblock {\em Bulletin of the London Mathematical Society}, 3(2):176--180,
  1971.
  
\bibitem[Gou97]{gouvea1997p}
Fernando~Q Gouv{\^e}a.
\newblock p-adic numbers.
\newblock In {\em p-adic Numbers}, pages 43--85. Springer, 1997.

\bibitem[Gro99]{gromov1999topological}
Misha Gromov.
\newblock Topological invariants of dynamical systems and spaces of holomorphic
  maps: I.
\newblock {\em Mathematical Physics, Analysis and Geometry}, 2(4):323--415,
  1999.

\bibitem[Hoo74]{Hood}
BM~Hood.
\newblock Topological entropy and uniform spaces.
\newblock {\em Journal of the London Mathematical Society}, 2(4):633--641,
  1974.

\bibitem[HR12]{HewittandRoss}
Edwin Hewitt and Kenneth~A Ross.
\newblock {\em Abstract Harmonic Analysis: Volume I Structure of Topological
  Groups Integration Theory Group Representations}, volume 115.
\newblock Springer Science \& Business Media, 2012.

\bibitem[HR15]{HuckandRichard}
Christian Huck and Christoph Richard.
\newblock On pattern entropy of weak model sets.
\newblock {\em Discrete \& Computational Geometry}, 54(3):741--757, 2015.

\bibitem[HYZ11]{HuangandYeandZhang}
Wen Huang, Xiangdong Ye, and Guohua Zhang.
\newblock Local entropy theory for a countable discrete amenable group action.
\newblock {\em Journal of Functional Analysis}, 261(4):1028--1082, 2011.


\bibitem[JLO16]{JaegerLenzandOertel}
Tobias J\"ager, Daniel Lenz, and Christian Oertel.
\newblock Model sets with positive entropy in euclidean cut and project
  schemes.
\newblock {\em Ann. Scient. \'Ec. Norm. Sup.}
52($4^e$)  1073 -- 1106, 2019.

\bibitem[Kel17]{Kelley}
John~L Kelley.
\newblock {\em General topology}.
\newblock Courier Dover Publications, 2017.

\bibitem[Kri10]{Krieger}
Fabrice Krieger.
\newblock The ornstein--weiss lemma for discrete amenable groups.
\newblock {\em Max Planck Institute for Mathematics Bonn, MPIM Preprint},
  48:2010, 2010.
  
  
\bibitem[LW00]{lindenstrauss2000mean}
Elon Lindenstrauss and Benjamin Weiss.
\newblock Mean topological dimension.
\newblock {\em Israel Journal of Mathematics}, 115(1):1--24, 2000.

\bibitem[Mac18]{machado2018approximate}
Simon Machado.
\newblock Approximate lattices and meyer sets in nilpotent lie groups.
\newblock {\em arXiv preprint arXiv:1810.10870}, 2018.




\bibitem[Mey72]{meyer1972algebraic}
Y~Meyer.
\newblock Algebraic numbers and harmonic analysis, vol. 2, 1972.

\bibitem[Mun00]{Munkres}
James~R Munkres.
\newblock {\em Topology}.
\newblock Prentice Hall, 2000.


\bibitem[Oll07]{ollagnier2007ergodic}
Jean~Moulin Ollagnier.
\newblock {\em Ergodic theory and statistical mechanics}, volume 1115.
\newblock Springer, 2007.

\bibitem[OP82]{ollagnier1982variational}
Jean~Moulin Ollagnier and Didier Pinchon.
\newblock The variational principle.
\newblock {\em Studia Mathematica}, 72:151--159, 1982.



\bibitem[OW87]{OrnsteinandWeiss}
Donald~S Ornstein and Benjamin Weiss.
\newblock Entropy and isomorphism theorems for actions of amenable groups.
\newblock {\em Journal d'Analyse Math{\'e}matique}, 48(1):1--141, 1987.

\bibitem[Pat98]{patera1998quasicrystals}
Jiri Patera.
\newblock {\em Quasicrystals and discrete geometry}, volume~10.
\newblock American Mathematical Soc., 1998.



\bibitem[Pat00]{paterson2000amenability}
Alan~LT Paterson.
\newblock {\em Amenability}.
\newblock Number~29. American Mathematical Soc., 2000.

\bibitem[Pie84]{Pier}
Jean-Paul Pier.
\newblock {\em Amenable locally compact groups}.
\newblock Wiley-Interscience, 1984.

\bibitem[PS16]{pogorzelski2016banach}
Felix Pogorzelski and Fabian Schwarzenberger.
\newblock A banach space-valued ergodic theorem for amenable groups and
  applications.
\newblock {\em Journal d'Analyse Math{\'e}matique}, 130(1):19--69, 2016.

\bibitem[Run04]{runde2004lectures}
Volker Runde.
\newblock {\em Lectures on amenability}.
\newblock Springer, 2004.

\bibitem[SBGC84]{shechtman1984metallic}
Dan Shechtman, Ilan Blech, Denis Gratias, and John~W Cahn.
\newblock Metallic phase with long-range orientational order and no
  translational symmetry.
\newblock {\em Physical review letters}, 53(20):1951, 1984.

\bibitem[Sch99]{Schlottmann}
Martin Schlottmann.
\newblock Generalized model sets and dynamical systems.
\newblock In {\em CRM Monograph Series}. Citeseer, 1999.

\bibitem[Sch15]{schneider2015topological}
Friedrich~Martin Schneider.
\newblock Topological entropy of continuous actions of compactly generated
  groups.
\newblock {\em arXiv preprint arXiv:1502.03980}, 2015.

\bibitem[Str15]{strungaru2015almost}
Nicolae Strungaru.
\newblock Almost periodic measures and meyer sets.
\newblock {\em Discrete Comput. Geom. 33}, no. 3, 483-505, 2005.

\bibitem[STZ80]{stepin1980variational}
Anatolii~Mikhailovich Stepin and AT~Tagi-Zade.
\newblock Variational characterization of topological pressure of the amenable
  groups of transformations.
\newblock In {\em Doklady Akademii Nauk}, volume 254, pages 545--549. Russian
  Academy of Sciences, 1980.

\bibitem[Tem]{Tempelman}
Arkadi{\u\i} Tempelman.
\newblock {\em Ergodic theorems for group actions: Informational and
  Thermodynamical Aspects}.

\bibitem[TZ91]{Tagi-Zade}
AT~Tagi-Zade.
\newblock Variational characterization of topological entropy of continuous
  transformation groups. case of actions of $\mathbb{R}^n$.
\newblock {\em Mathematical notes of the Academy of Sciences of the USSR},
  49(3):305--311, 1991.
  
 

\bibitem[Wei03]{weiss2003actions}
Benjamin Weiss.
\newblock Actions of amenable groups.
\newblock {\em Topics in dynamics and ergodic theory}, 310:226--262, 2003.

\bibitem[WZ92]{ward1992abramov}
Thomas Ward and Qing Zhang.
\newblock The abramov-rokhlin entropy addition formula for amenable group
  actions.
\newblock {\em Monatshefte f{\"u}r Mathematik.}, 114(3-4):317--329, 1992.


\bibitem[Yan15]{Yan}
Kesong Yan.
\newblock Conditional entropy and fiber entropy for amenable group actions.
\newblock {\em Journal of Differential Equations}, 259(7):3004--3031, 2015.

\bibitem[YZ16]{YanandZeng}
Kesong Yan and Fanping Zeng.
\newblock Topological entropy, pseudo-orbits and uniform spaces.
\newblock {\em Topology and its Applications}, 210:168--182, 2016.

\bibitem[ZC16]{ZhengandEChen}
Dongmei Zheng and Ercai Chen.
\newblock Bowen entropy for actions of amenable groups.
\newblock {\em Israel Journal of Mathematics}, 212(2):895--911, 2016.

\bibitem[Zho16]{zhou2016tail}
Yunhua Zhou.
\newblock Tail variational principle for a countable discrete amenable group
  action.
\newblock {\em Journal of Mathematical Analysis and Applications},
  433(2):1513--1530, 2016.




\end{thebibliography}
\newcommand{\etalchar}[1]{$^{#1}$}

 \vspace{10mm} \noindent
\begin{tabular}{l l }
Till Hauser \\
Faculty of Mathematics and Computer Science\\
Institute of Mathematics \\
Friedrich Schiller University Jena\\
07743 Jena\\
Germany\\
\end{tabular}

\end{document}